\documentclass{article}%
\usepackage[margin=1.65 in]{geometry}
\usepackage{amsmath}
\usepackage[linkcolor=blue,urlcolor=blue,colorlinks=true]{hyperref}
\usepackage{amsfonts}
\usepackage{amssymb}
\usepackage{graphicx}
\usepackage{diagbox}
\usepackage{subfig}
\usepackage{epsfig}%
\usepackage{color}
\usepackage{multirow}
\usepackage{setspace}
\usepackage{hyperref}
\usepackage{enumitem}
\usepackage{hhline}
\usepackage{pbox}
\usepackage{booktabs}
\usepackage{array}
\usepackage[flushleft]{threeparttable}
\usepackage{authblk}
\usepackage{float}
\usepackage{supertabular}
\usepackage[title]{appendix}
\newtheorem{theorem}{Theorem}

\newtheorem{corollary}[theorem]{Corollary}

\newtheorem{lemma}[theorem]{Lemma}

\newenvironment{proof}[1][Proof]{\noindent\textbf{#1.} }{\ \rule{0.5em}{0.5em}}
\begin{document}
\doublespacing

\title{\textbf{A Bayesian Nonparametric Test for Assessing Multivariate Normality}}   



\author[1]{Luai Al-Labadi\thanks{{\em Corresponding author:} luai.allabadi@utoronto.ca}}

\author[2]{Forough Fazeli Asl\thanks{forough.fazeli@math.iut.ac.ir}}

\author[2]{Zahra Saberi\thanks{ z\_saberi@cc.iut.ac.ir}}

\affil[1]{Department of Mathematical and Computational Sciences, University of Toronto Mississauga, Mississauga, Ontario L5L 1C6, Canada.}
\affil[2]{Department of Mathematical Sciences, Isfahan University of Technology, Isfahan 84156-83111, Iran.}

\date{}
\maketitle

\pagestyle {myheadings} \markboth {} {A BNP Test for Multivariate Normality}

\begin{abstract}
A novel Bayesian nonparametric test for assessing multivariate normal models is presented. While there are extensive frequentist and graphical methods for testing multivariate normality, it is challenging to find Bayesian counterparts. The approach is based on the Dirichlet process and Mahalanobis distance. Specifically, the Mahalanobis distance is employed to transform the $m$-variate problem into a univariate problem. Then the Dirichlet process is used as a prior on the distribution of the distance. The concentration of the distribution of the distance between the posterior process and the chi-square distribution with $m$ degrees of freedom is compared to that between the prior process and the chi-square distribution via a relative belief ratio. The distance between the Dirichlet process and the chi-square distribution is established based on the Anderson-Darling distance. Key results of the approach are derived. The procedure is illustrated through several examples, in which it shows excellent performance. 
\par

 \vspace{9pt} \noindent\textsc{Keywords:}  Anderson-Darling distance, Dirichlet process, Mahalanobis distance, Multivariate hypothesis testing, Relative belief inferences.

 \vspace{9pt}


\end{abstract}
	
\section{Introduction}
The assumption of multivariate normality is a key assumption in many statistical applications such as pattern recognition and exploratory multivariate methods \cite{Dubes,Fernandez}. The need to check this assumption is of special importance as if it does not hold, the obtained results based on this assumption may lead to an error.
Specifically, for a given $m$-variate sampled data $\mathbf{y}_{m\times n}=\left(
\mathbf{y}_{1},\ldots,\mathbf{y}_{n}\right)$
with size of $n$, where $\mathbf{y}_{i}\in\mathbb{R}^{m}$ for $i=1,\ldots,n$, the interest is to  determine whether $\mathbf{y}_{m\times n}$ comes from a multivariate normal population.

Many tests and graphical methods have been considered to assess the multivariate normality assumption. Healy \cite{Healy} described an extension of normal plotting techniques to handle multivariate data. Mardia \cite{Mardia}
proposed a test based on the asymptotic distribution of measures of multivariate skewness and kurtosis. Tests based on transforming the multivariate problem into the univariate problem were established by Rinc\'{o}n-Gallardo et al.\cite{Rincon}, Royston \cite{Royston}, Fattorini \cite{Fattorini}, and Hasofer and Stein\cite{Hasofer}. A class of
invariant consistent tests based on a weighted integral of the squared distance between the empirical characteristic function and the characteristic function of the multivariate normal distribution was suggested by  Henze and Zirkler \cite{Henze1990}.  Holgersson \cite{Holgersson} presented a simple graphical method based on the scatter plot. Doornik and Hansen \cite{Doornik} developed an omnibus test based on a transformed skewness and kurtosis. They showed that their test is more powerful than the Shapiro-Wilk test proposed by  Royston \cite{Royston}.  Alva and Estrada \cite{Villasenor} proposed a multivariate normality test based on Shapiro-Wilk’s statistic for univariate normality and on an empirical standardization
of the observations. J\"{o}nsson \cite{Jonsson} presented a robust test with powerful performance.  Hanusz and Tarasi\'{n}ska \cite{Hanusz} proposed two  tests based on Small's and Srivastava's graphical methods.
A powerful test was offered by Zhou and Shao \cite{Zhou} with good application in biomedical studies.  Kim \cite{Kim} generalized Jarque-Bera univariate normality test to the multivariate case. Kim and Park \cite{Kim2018} derived several powerful omnibus tests based on the likelihood ratio and the characterizations of the multivariate normal distribution. Madukaife and Okafor \cite{Madukaife} proposed a powerful affine invariant test  based on interpoint distances of principal components. Henze and Visagie \cite{Henze2019} derived a new test based on a partial differential equation involving the moment generating function.

While there are considerable frequentist  and  graphical methods for testing multivariate normality, it is difficult  to find relevant Bayesian counterparts. Most available Bayesian tests focused on employing Bayesian nonparametric methods for univariate data. See for example  Al-Labadi and Zarepour \cite{Al-Labadi2013,Al-Labadi2014a,Al-Labadi2017a}. A remarkable work that covers the multivariate case was developed by Tokdar and Martin \cite{Tokdar}, where they established a Bayesian test based on characterizing alternative models by Dirichlet process mixture distributions. The authors took the Bayes factor as evidence for or against the null distribution, using an importance sampling for inference. In the current paper, a novel Bayesian nonparametric test for assessing multivariate normality is proposed. The developed test is based on using Mahalanobis distance. Using this distance has the advantage of converting the $m$-variate problem into a univariate problem. Specifically, whenever $\mathbf{y}_{m\times n}$ comes from a multivariate normal distribution, the distribution of the corresponding Mahalanobis distances, denoted by $P$, is approximately chi-square with $m$ degrees of freedom \cite{Johnson}. This reduces the problem to test the hypothesis $\mathcal{H}_{0}:P=F_{(m)}$, where $F_{(m)}$ denotes the cumulative distribution function of the chi-square distribution with $m$ degrees of freedom. The proposed test deems the Dirichlet process as a prior on $P$. Then the concentration of the distribution of the distance between the posterior process and $F_{(m)}$ is compared to the concentration of the distribution of the distance between the prior process and $F_{(m)}$. The distance between the Dirichlet process and $F_{(m)}$ is developed based on the Anderson-Darling distance as an appropriate tool to detect the difference, especially when this difference is due to the tails. This comparison is made via a \textit{relative belief ratio}. A calibration of the relative belief ratio is also presented.  The anticipated test is easy to implement with excellent performance, it does not require providing a closed form of the relative belief ratio. In addition, unlike the tests that use p-values, the proposed test permits us to state evidence for the null hypothesis. We point out that the comparison between the concentration of the posterior and the prior distribution of the distance was suggested by Al-Labadi and Evans \cite{Al-Labadi2018} for model checking of univariate distributions.

The remainder of this paper is organized as follows.  A general discussion about the relative belief ratio is given in Section 2. The definition and some fundamental properties of the Dirichlet process are presented in Section 3. An explicit expression for calculating the Anderson-Darling distance between the Dirichlet process and a continuous distribution is derived in Section 4. A Bayesian nonparametric test for assessing multivariate normality and some of its relevant properties are developed in Section 5. A computational algorithm to carry out the proposed test is discussed in Section 6. The performance of the proposed test is discussed in Section 7, where several simulated examples and a real data set are considered.  Finally, a summary of the findings is given in Section 8. All technical proofs are included in the Appendix.
\section{Relative belief inferences}
The relative belief ratio is a common measure of statistical evidence. It leads to a straightforward inference in hypothesis testing problems. For more details, let $\{f_{\theta}:\theta\in\Theta\}$ denote a collection of densities on a sample space $\mathfrak{X}$ and let $\pi$ denote a prior on the parameter space $\Theta$. Note that
the densities may represent discrete or continuous probability measures but they are
all with respect to the same support measure  $d\theta$. After
observing the data $x,$ the posterior distribution of $\theta$, denoted by $\pi(\theta\,|\,x)$, is a revised prior and is given by the
density $\pi(\theta\,|\,x)=\pi(\theta)f_{\theta}(x)/m(x)$, where $m(x)=\int
_{\Theta}\pi(\theta)f_{\theta}(x)\,d\theta$ is the prior predictive density of
$x.$  For a parameter of interest $\psi=\Psi(\theta),$ let $\Pi_{\Psi}$ denote
the marginal prior probability measure and $\Pi_{\Psi}(\cdot|\,x)$ denote
the marginal posterior probability measure. It is assumed that
$\Psi$ satisfies regularity conditions
so that the prior density $\pi_{\Psi}$ and the posterior density
$\pi_{\Psi}(\cdot\,|\,x)$ of $\psi$ exist with respect to some support measure on the range space for $\Psi$
\cite{Evans2015}. The relative belief ratio for a value
$\psi$ is then defined by $$RB_{\Psi}(\psi\,|\,x)=\lim_{\delta\rightarrow0}%
\Pi_{\Psi}(N_{\delta}(\psi\,)|\,x)/\Pi_{\Psi}(N_{\delta}(\psi\,)),$$ where
$N_{\delta}(\psi\,)$ is a sequence of neighborhoods of $\psi$ converging
nicely to $\psi$ as $\delta\rightarrow0$ \cite{Evans2015}. When $\pi_{\Psi}$ and  $\pi_{\Psi}(\cdot\,|\,x)$ are continuous at $\psi,$ the relative belief ratio is defined by
\begin{equation*}
RB_{\Psi}(\psi\,|\,x)=\pi_{\Psi}(\psi\,|\,x)/\pi_{\Psi}(\psi), \label{relbel}%
\end{equation*}
the ratio of the posterior density to the prior density at $\psi.$  Therefore,
$RB_{\Psi}(\psi\,|\,x)$ measures the change in the belief of $\psi$ being the true value from a \textit{priori} to a \textit{posteriori}. Note that a relative belief ratio is similar to a Bayes factor, as both are measures of evidence, but the latter measures evidence via the change in an odds ratio. In general, when a Bayes factor is defined via a limit in the
continuous case, the limiting value equals the corresponding relative belief ratio. For a further discussion about the relationship between relative belief ratios and Bayes factors
see, for instance,  Evans \cite{Evans2015}, Chapter 4.

Since $RB_{\Psi}(\psi\,|\,x)$ is a measure of the evidence that $\psi$ is the true value, if $RB_{\Psi}(\psi\,|\,x)$ $>1$, then the probability of the $\psi$ being the true value from a priori to a posteriori is increased, consequently there is evidence based on the data that $\psi$ is the true value. If $RB_{\Psi}(\psi\,|\,x)<1$, then the probability of the $\psi$ being the true value from a priori to a posteriori is decreased. Accordingly, there is evidence against based on the data that $\psi$ being the true value. For the case $RB_{\Psi}(\psi\,|\,x)=1$ there is no
evidence either way.

Obviously, $RB_{\Psi}(\psi_{0}\,|\,x)$ measures the evidence of the hypothesis $\mathcal{H}_{0}:\Psi(\theta)=\psi_{0}$. For a large  value of $c$ $(c\gg 1)$, $RB_{\Psi}(\psi_{0}\,|\,x)=c$ provides
 strong evidence in favor of $\psi_{0}$ because belief has increased by a
factor of $c$ after seeing the data. However, there may also exist other
values of $\psi$ that had even larger increases. Thus, it is also necessary, however, to calibrate whether this is strong or weak evidence for
or against $\mathcal{H}_{0}.$ A typical
calibration of $RB_{\Psi}(\psi_{0}\,|\,x)$  is given by the  \textit{strength}
\begin{equation}
\Pi_{\Psi}\left[RB_{\Psi}(\psi\,|\,x)\leq RB_{\Psi}(\psi_{0}\,|\,x)\,|\,x\right].
\label{strength}%
\end{equation}
The value in \eqref{strength} indicates that the posterior probability that the true value of $\psi$ has a relative
belief ratio no greater than that of the hypothesized value $\psi_{0}.$ Noticeably, (\ref{strength}) is not a p-value as it has a very different
interpretation. When $RB_{\Psi}(\psi_{0}\,|\,x)<1$, there is evidence
against $\psi_{0},$ then a small value of (\ref{strength}) indicates
 strong evidence against $\psi_{0}$. On the other hand, a large value for \eqref{strength}    indicates   weak evidence against $\psi_{0}$.
Similarly, when $RB_{\Psi}(\psi_{0}\,|\,x)>1$, there is  evidence in favor
of $\psi_{0},$ then a small value of (\ref{strength}) indicates  weak
evidence in favor of $\psi_{0}$, while a large value of \eqref{strength} indicates
 strong evidence in favor of $\psi_{0}$. For applications of the use of the relative belief ratio in different univariate hypothesis testing problems, see  Evans \cite{Evans1997,Evans2015}, Al-Labadi and Evans \cite{Al-Labadi2018} and  Al-Labadi et al. \cite{Al-Labadi-Baskurt2017,Al-Labadi-Baskurt2018}.
 \section{Dirichlet process}

The Dirichlet process, introduced by Ferguson \cite{Ferguson73}, is the
most commonly used prior in the Bayesian nonparametric inferences.  In this section,  we recall the most relevant properties of this prior. Consider a space
$\mathfrak{X}$ with a $\sigma$-algebra $\mathcal{A}$ of subsets of
$\mathfrak{X}$, let $H$ be a fixed probability measure on $(\mathfrak{X}%
,\mathcal{A}),$ called the \emph{base measure}, and $a$ be a positive number,
called the \emph{concentration parameter}. A random probability measure
$P=\left\{  P(A):A\in\mathcal{A}\right\}  $ is called a Dirichlet process on
$(\mathfrak{X},\mathcal{A})$ with parameters $a$ and $H,$ denoted by $P\sim
{DP}(a,H),$ if for every measurable partition $A_{1},\ldots,A_{k}$ of
$\mathfrak{X} $ with $k\geq2\mathfrak{,}$ the joint distribution of the vector
$\left(  P(A_{1}),\ldots\,P(A_{k})\right)$ has the Dirichlet distribution with parameter
$aH(A_{1}),\ldots,$ $aH(A_{k})$. Also, it is assumed that
$H(A_{j})=0$ implies $P(A_{j})=0$ with probability one. Consequently, for any
$A\in\mathcal{A}$, $P(A)\sim$ beta$(aH(A),a(1-H(A)))$,
${E}(P(A))=H(A)\ $and ${Var}(P(A))=H(A)(1-H(A))/(1+a).$ Accordingly, the base measure $H$
plays the role of the center of $P$ while the concentration parameter $a$ controls variation of
the process $P$ around the base measure $H$. One of the most well-known
properties of the Dirichlet process is the conjugacy property. That is, when the sample $x=(x_{1},\ldots,x_{n})$
is drawn from $P\sim DP(a,H)$, the posterior distribution of $P$ given $x$,
denoted by $P_{x}$, is also a Dirichlet process with
concentration parameter $a+n$ and base measure
\begin{equation}\label{pos base measure}
H_{x}=a(a+n)^{-1}H+n(a+n)^{-1}F_{n},
\end{equation}
where $F_{n}$ denotes the empirical cumulative distribution function (cdf) of the sample
$x$. Note that, $H_{x}$ is a convex combination
of the base measure $H$ and the empirical cdf $F_{n}$. Therefore, $H_{x}\rightarrow H$ as
$a\rightarrow\infty$ while $H_{x}\rightarrow F_{n}$ as $a\rightarrow0.$ A detailed discussion
about choosing the hyperparameters $a$ and $H$ will be presented in Section 5.

Following Ferguson \cite{Ferguson73}, $P\sim{DP}(a,H)\ $can be represented as
\begin{equation}
P=\sum_{i=1}^{\infty}L^{-1}(\Gamma_{i}){\delta_{Y_{i}}/}\sum_{i=1}^{\infty
}{{L^{-1}(\Gamma_{i})}}, \label{series-dp}%
\end{equation}
where $\Gamma_{i}=E_{1}+\cdots+E_{i}$ with $E_{i}\overset{i.i.d.}{\sim}%
$\ exponential$(1),Y_{i}\overset{i.i.d.}{\sim}H$ independent of the
$\Gamma_{i},L^{-1}(y)=\inf\{x>0:L(x)\geq y\}$ with $L(x)=a\int_{x}^{\infty
}t^{-1}e^{-t}dt,x>0,$ and ${\delta_{a}}$ the Dirac delta measure. The series representation
(\ref{series-dp}) implies that the Dirichlet process is a discrete probability
measure even for the cases with an absolutely continuous base measure $H$. Note that, since data is always measured
to finite accuracy, the true distribution being sampled from is discrete. This makes
the discreteness property of $P$ has no significant limitation. By imposing the weak topology, the support of the Dirichlet process could
be quite large. To be more precise, when the support of the base measure
is $\mathfrak{X}$, then the space of all probability measures on $\mathfrak{X}$  is the support of
the Dirichlet process.

Because no closed form exists for $L^{-1}(\cdot)$, working with representation \eqref{series-dp} to sample from the Dirichlet process is not simple. This issue motivates researchers to propose several methods in order to
approximately simulate the Dirichlet process. One such efficient method was presented by Zarepour and Al-Labadi \cite{Zarepour2012}. They showed that the Dirichlet process $P\sim DP(a,H)$ can be approximated by
\begin{equation}\label{approx of DP}
P_{N}=\sum_{i=1}^{N}J_{i}\delta_{Y_{i}},
\end{equation}
with the
monotonically decreasing weights
$J_{i}=\frac{G_{a/N}^{-1}(\frac{\Gamma_{i}}{\Gamma_{N+1}})}{\sum_{j=1}^{N}G_{a/N}^{-1}(\frac{\Gamma_{i}}{\Gamma_{N+1}})},$
where $\Gamma_{i}$ and $Y_{i}$ are defined as before, $N$ is a positive large
integer and $G_{a/N}$ denotes the complement-cdf of the $\text{gamma}(a/N,1)$ distribution.
Note that, $G^{-1}_{a/N}(p)$ is the $(1-p)$-th quantile of the $\text{gamma}(a/N,1)$ distribution.
Also, Zarepour and Al-Labadi \cite{Zarepour2012} showed that $P_{N}$ converges almost surely to
(\ref{series-dp}) as $N$ goes to infinity with a fast rate of convergence. The following
algorithm describes how the approximation \eqref{approx of DP} can be used to generate
a sample from $DP(a,H)$.

\noindent \textbf{Algorithm A} (\emph{Approximately generating a sample from }%
$DP(a,H))$:

\begin{enumerate}
\item[$i.$] Fix a large positive integer $N$ and generate i.i.d. $Y_{i}\sim H$ for $i=1,\ldots,N$.

\item[$ii.$] For $i=1,\ldots,N+1$, generate i.i.d. $E_{i}$ from the exponential distribution with rate 1, independent
of $\left(Y_{i}\right)_{1\leq i\leq N}$ and put $\Gamma_{i}=E_{1}+\cdots+E_{i}$.
\item[$iii.$] Compute $G_{a/N}^{-1}\left(  {\Gamma_{i}}/{\Gamma_{N+1}}\right)  $ for
$i=1,\ldots,N$ and return $P_{N}.$
\end{enumerate}

\noindent For other  simulation methods of the Dirichlet process, see Bondesson
\cite{Bondesson}, Sethuraman \cite{Sethuraman}, Wolpert and Ickstadt \cite{Wolpert} and Al-Labadi and Zarepour  \cite{Al-Labadi2014b}.
\section{Measuring the distance}

Measuring the distance between two distributions is an essential  tool in model checking. In this section, two well-known distances, namely \textit{Anderson-Darling} distance and \textit{Mahalanobis} distance, are considered.

\subsection{Anderson-Darling distance}
The Anderson-Darling distance between two cdf's $F$ and $G$ is given by
\begin{equation*}
d_{AD}(F,G)=\int_{-\infty}^{\infty}\dfrac{\left(F(t)-G(t)\right)^{2}}{G(t)\left(1-G(t)\right)}\, dG(t).
\end{equation*}
Anderson-Darling distance can be viewed as a modification of the Cram\'{e}r-von
Mises distance that gives more weight to data points in the tails of the distribution.
The next lemma provides an explicit formula to compute the Anderson-Darling distance
between a Dirichlet process and a continuous cdf. Throughout this paper, $``\log"$ denotes the natural logarithm.

\begin{lemma}\label{discrete AD}
Let $G$ be a continuous cdf and $P_{N}=\sum_{i=1}^{N}J_{i}\delta_{Y_{i}}$ be a
discrete distribution, where $(Y_{i})_{1\leq i\leq N}$ are $i.i.d.$ from $H$. Let $Y_{(1)}\leq\cdots \leq Y_{(N)}$ be the order statistics of
$(Y_{i})_{1\leq i\leq N}$ and $J_{1}^{\prime},\ldots ,J_{N}^{\prime}$ are the
associated jump sizes such that $J_{i}=J_{j}^{\prime}$ when $Y_{i}=Y_{(j)}$.
Then the Anderson-Darling distance between $P_{N}$ and $G$ is given by
\begin{small}

\begin{align*}\label{app of AD-1}
d_{AD}(P_{N},G)&=2\sum_{i=1}^{N-1}\sum_{j=1}^{i-1}
\sum_{k=j+1}^{i}J_{j}^{\prime}J_{k}^{\prime}C_{i,i+1}+
\sum_{i=1}^{N-1}\sum_{j=1}^{i}J_{j}^{\prime ^{2}}C_{i,i+1}+2\sum_{i=1}^{N-1}
\sum_{j=1}^{i}J_{j}^{\prime}C_{i,i+1}^{\ast}\nonumber\\
&-\sum_{i=1}^{N-1}C_{i,i+1}^{\ast}-\log [G(Y_{(N)})(1-G(Y_{(1)}))]-1,
\end{align*}

\end{small}
where $C_{i,i+1}=\log \frac{G(Y_{(i+1)})\left(1-G(Y_{(i)})\right)}{G(Y_{(i)})\left(1-G(Y_{(i+1)}
)\right)}$ and $C_{i,i+1}^{\ast}=\log\frac{1-G(Y_{(i+1)})}{1-G(Y_{(i)})}$.
\end{lemma}

\begin{proof}
The proof is given in Appendix A.
\end{proof}

The next corollary indicates that the distribution of $d_{AD}(P_{N},G)$ is independent of $G$ when $G=H$ and $P_{N}=\sum_{i=1}^{N}J_{i}\delta_{Y_{i}}$.

\begin{corollary}\label{corollary}
Suppose that $(Y_{i})_{1\leq i\leq N}$ are $i.i.d.$ from $G$, independent of
$(J_{i})_{1\leq i\leq N}$ and $P_{N}=\sum_{i=1}^{N}J_{i}\delta_{Y_{i}}$. Then
\begin{small}
$d_{AD}(P_{N},G)\overset{d}{=}
2\sum_{i=1}^{N-1}\sum_{j=1}^{i}J_{j}^{\prime}U_{i,i+1}^{\ast}
-\sum_{i=1}^{N-1}U_{i,i+1}^{\ast}\\
+\sum_{i=1}^{N-1}
\sum_{j=1}^{i}J_{j}^{\prime ^{2}}U_{i,i+1}+2\sum_{i=1}^{N-1}\sum_{j=1}^{i-1}
\sum_{k=j+1}^{i}J_{j}^{\prime}J_{k}^{\prime}U_{i,i+1}
-\log [U_{(N)}(1-U_{(1)})]-1$\end{small}, where
$U_{i,i+1}=\log \frac{U_{(i+1)}\left(1-U_{(i)}\right)}{U_{(i)}\left(1-U_{(i+1)}\right)}$,
$U_{i,i+1}^{\ast}=\log \frac{1-U_{(i+1)}}{1-U_{(i)}}$ and $U_{(i)}$ is the
i-th order statistic for $(U_{i})_{1\leq i\leq N}$ i.i.d. uniform$[0, 1]$.
\end{corollary}

\begin{proof}
Since for $1\leq i\leq N$, $U_{i}=G(Y_{i})$ and $Y_{i}$ is a sequence of i.i.d.
random variable with continuous distribution $G$, the probability integral transformation
theorem implies that $(U_{i})_{1\leq i\leq N}$ is a sequence of i.i.d. random variable
uniformly distributed on the interval $[0,1]$. The proof is completed by considering
the order statistic of $(U_{i})_{1\leq i\leq N}$ in Lemma \ref{discrete AD}.
\end{proof}



\subsection{Mahalanobis distance}

Mahalanobis distance measures the distance of  $m$-variate
point $\mathbf{Y}$ generated from a known distribution $F_{\theta}$ to the
mean $\boldsymbol{\mu}_{m} =E_{\theta}(\mathbf{Y})$ of the distribution. Let $\Sigma_{m\times m}$ be the
covariance matrix of the $m$-variate distribution $F_{\theta}$, the Mahalanobis
distance is defined as
\begin{equation}\label{Mah}
D_{M}(\mathbf{Y})=\sqrt{(\mathbf{Y}-\boldsymbol{\mu}_{m})^{T}\Sigma^{-1}_{m\times m}
(\mathbf{Y}-\boldsymbol{\mu}_{m})}.
\end{equation}

Note that, \eqref{Mah} is limited to the cases when both $\boldsymbol{\mu}_{m}$
and $\Sigma_{m\times m}$ are known. However, in most cases, the mean and covariance of $F_{\theta}$ exist but
unknown. For such cases, the Mahalanobis distance is defined based on the measuring
of distance between a subject's data and the mean of all observations in an observed sample.
To be more precise, given a sample of $n$ independent $m$-variate
$\mathbf{y}_{1},\ldots , \mathbf{y}_{n}$, the sample Mahalanobis distance of $\mathbf{y}_{i}$ to the sample mean $\overline{\mathbf{y}}$,
denoted by $d_{M}(\mathbf{y}_{i})$, is defined by
\begin{equation}\label{sample-mah}
d_{M}(\mathbf{y}_{i})=\sqrt{(\mathbf{y}_{i}-\overline{\mathbf{y}})^{T}S^{-1}_{y}(\mathbf{y}_{i}-\overline{\mathbf{y}})},
\end{equation}
where $S_{y}$ denotes the sample covariance matrix. Some interesting properties of the
Mahalanobis distance are presented in \cite{Johnson}. For example, when the
parent population is $m$-variate normal, for large enough values of $n$ and $n-m$
($n, n-m>30$), each of the squared distance $d_{M}^{2}(\mathbf{y}_{1}), \ldots
, d_{M}^{2}(\mathbf{y}_{n})$ behave like a chi-square random variable with $m$ degrees of
freedom.
\section{ Bayesian nonparametric approach for assessing multivariate normality}

In this section, a new test for assessing multivariate normality
is presented. For this purpose, consider the family of $m$-variate normal distribution
$\mathcal{F}=\{N_{m}(\boldsymbol{\mu}_{m},\Sigma_{m\times m}):\boldsymbol{\mu}_{m}\in\mathbb{R}^{m}
,\det (\Sigma_{m\times m}) >0\}$. Let $\textbf{Y}_{1},\ldots,\textbf{Y}_{n}$ be a random sample
from $m$-variate distribution $F$. The  problem under consideration is to test the hypothesis
\begin{equation}\label{test1}
\mathcal{H}_{0}:F\in \mathcal{F},
\end{equation}
using the Bayesian nonparametric framework. The first step is to
reduce the  multivariate problem to a univariate
problem. One way to accomplish that is through the Mahalanobis distance. For this, assume that $\boldsymbol{\widehat{\mu}}_{m}=\overline{\mathbf{Y}}$ and $\widehat{\Sigma}_{m\times m}=S_{Y}$ are the sample mean and sample covariance matrix, respectively. Then, $N_{m}(\widehat{\boldsymbol\mu}_{m},\widehat{\Sigma}_{m\times m})$
is the best representative of the family $\mathcal{F}$ to compare with distribution $F.$ Define $$D_{M}(\mathbf{Y}_{i})=\sqrt{(\mathbf{Y}_{i}-\boldsymbol{\widehat{\mu}}_{m})^{T}
\widehat{\Sigma}_{m\times m}^{-1}(\mathbf{Y}_{i}-\boldsymbol{\widehat{\mu}}_{m})} \ \ \ \ \text { for }1\leq i
\leq n.$$
Assume that $F_{(m)}$ is the cdf of the chi-square
distribution with $m$ degrees of freedom and $(D_{M}^{2}(\mathbf{Y}_{i}))_{1\leq i
\leq n}$ is a sequence of random variables with continuous distribution function $P$.
From  Johnson and Wichern \cite{Johnson}, if $\mathcal{H}_{0}$ is true, then we expect that $P$ is (approximately) the same as $F_{(m)}$. Thus, testing \eqref{test1} is equivalent to testing

\begin{equation}\label{test2}
\mathcal{H}_{0}:P=F_{(m)}.
\end{equation}

For testing \eqref{test2}, let $\mathbf{y}_{m\times n}=(\mathbf{y}_{1},\ldots ,\mathbf{y}_{n})$ be
an observed sample from $F$ and $d=(d_{M}^{2}(\mathbf{y}_{1}),\ldots,
d_{M}^{2}(\mathbf{y}_{n}))$ be the corresponding observed squared Mahalanobis distance.
If $P\sim DP(a,F_{(m)})$, for a given choice of $a$, by \eqref{pos base measure}, then $P_{d}=P|d\sim DP(a+n, H_{d})$.
From \eqref{sample-mah},
if $\mathcal{H}_{0}$ is true, then the sampled observations $d_{M}^{2}(\mathbf{y}_{1}),\ldots, d_{M}^{2}
(\mathbf{y}_{n})$ would be approximately independent chi-square with $m$ degrees of
freedom. Thus, the posterior
distribution of the distance between $P_{d}$ and $F_{(m)}$ should be more concentrated
around zero than the prior distribution of the distance between $P$ and $F_{(m)}$.
A good reason for choosing $H=F_{(m)}$ is to avoid prior-data
conflict. Prior-data conflict means that there is a tiny overlap between the effective support regions of
$P$ and $H$. It can lead to errors in the computation of the prior distribution
of the distance between $P$ and $F_{(m)}$ and then the value of the relative belief ratio and its relevant strength when $\mathcal{H}_{0}$ is not true, see for example Section 7, Table \ref{prior-data}. For more discussion
about prior-data conflict see \cite{Evans2006,Al-Labadi2017b,Al-Labadi2018,Nott}.

In the proposed test, we use Lemma \ref{discrete AD} to compute the Anderson-Darling distances
$d_{AD}(P_{N},F_{(m)})$
and $d_{AD}(P_{d_{N}},F_{(m)})$, where $P_{N}$ and $P_{d_{N}}$ are approximations of Dirichlet
processes $P\sim DP(a, H)$ and $P_{d}\sim DP(a+n, H_{d})$, respectively. Then, the relative belief ratio is used to compare
the concentration of the posterior and the prior distribution of $d_{AD}(P_{d_{N}},F_{(m)})$
and $d_{AD}(P_{N},F_{(m)})$, respectively, at zero.

Another significant concern to validate the test is determining suitable values of $a$.  The following Lemma presents the expectation and variance of the prior distribution of the Anderson-Darling distance. For the Cram\'{e}r-von  Mises distance $d_{CvM}$, Al-Labadi and Evans \cite{Al-Labadi2018} showed that $E_{P}\left( d_{CvM}(P, H)\right)=1/6(a+1)$.

\begin{lemma}\label{Exp-Var of AD}
Let $P\sim DP(a,H)$ and $d_{AD}(P, H)$ be the Anderson-Darling distance between $P$ and
$H$, then
\begin{enumerate}
\item[$i.$] $E_{P}\left( d_{AD}(P, H)\right)=\dfrac{1}{a+1}$,
\item[$ii.$] $Var_{P}\left( d_{AD}(P, H)\right)=\dfrac{2\left( (\pi^{2}-9)a^{2}+(30-2\pi^{2})a-
3\pi^{2}+36\right)}{3(a+1)^{2}(a+2)(a+3)}$,
\end{enumerate}
where $E_{P}\left( d_{AD}(P, H)\right)$ and $V_{P}\left( d_{AD}(P, H)\right)$ are the expectation and variance of
the prior distribution of $d_{AD}(P, H)$ with respect to $P$, respectively.
\end{lemma}

\begin{proof}
The proof is given in Appendix B.
\end{proof}

The next corollary highlights the effect of the value $a$ on the prior distribution of the Anderson-Darling
distance.

\begin{corollary}\label{limit of Exp-Var}
Let $P\sim DP(a,H)$ and $H=F_{(m)}$ be the cdf of the chi-square
distribution with $m$ degrees of freedom. Suppose that $a\rightarrow \infty$, then
\begin{enumerate}
\item[$i$] $E_{P}\left( d_{AD}(P, H)\right)\rightarrow 0$, and,
$Var_{P}\left( d_{AD}(P, H)\right)\rightarrow 0$.
\item[$ii.$] \mbox{ }$E_{P}\left((a+1)d_{AD}(P, H)\right)\rightarrow 1$, and,
$Var_{P}\left((a+1)d_{AD}(P, H)\right)\rightarrow\dfrac{2}{3}(\pi^2-9)$.
\item[$iii.$] \mbox{ }$d_{AD}\xrightarrow{qm}0$, and, $d_{AD}\xrightarrow{a.s}0$.
\end{enumerate}

Where $``\xrightarrow{qm}"$ denotes convergence in quadratic mean.
\end{corollary}

\begin{proof}
The proof of (i) and (ii) are followed immediately by letting $a\rightarrow\infty$ in part (i) and (ii) of Lemma
\ref{Exp-Var of AD}. For (iii), the convergence in quadratic mean is followed by letting $a\rightarrow\infty$ in
$E_{P}\left( d_{AD}(P, H)\right)^{2}=\frac{a(2\pi^{2}-15)-6(\pi^{2}-15)}{3(a+3)(a+2)(a+1)}$; see the proof of Lemma \ref{Exp-Var of AD} in Appendix B. To prove the almost surely convergence,
assume that $a=kc$, for $k\in\lbrace 1, 2,\ldots,\rbrace$ and a fixed positive number $c$.
Then, for any $\epsilon>0$, $Pr\left\lbrace d_{AD}(P, H)>\epsilon\right\rbrace\leq\frac{E_{P}\left( d_{AD}(P, H)\right)^{2}}{\epsilon^{2}}$. Since, $\sum_{k=1}^{\infty}Pr\left\lbrace d_{AD}(P, H)>\epsilon\right\rbrace<\infty$, then, by the first Borel-Catelli lemma, $d_{AD}\xrightarrow{a.s}0$, as $k\rightarrow\infty$ ($a\rightarrow\infty$).
\end{proof}

\noindent Interestingly, the limit of the expectation and variance in part (ii) of Corollary
\ref{limit of Exp-Var} coincide with the limit of $E_{P}\left(n\, d_{AD}(F_{n}, H)\right)$ and
$Var_{P}\left(n\, d_{AD}(F_{n}, H)\right)$, given by Anderson and Darling \cite{Anderson},
 as $n\rightarrow\infty$, where $F_{n}$ is the
empirical cdf of the sample $d$.

According to part (i) of Corollary \ref{limit of Exp-Var}, with increasing
the value of $a$, the prior distribution of the distance becomes more concentrated about 0. As recommended
in Al-Labadi and Zarepour \cite{Al-Labadi2017a}, the value of $a$ should be at most $0.5\,n$, where
$n$ is the sample size; otherwise, the prior may become too influential.
On the other hand, if $\mathcal{H}_{0}$ is true,  with increasing the sample size $n$, the expectation of the posterior distribution
of the distance between $P_{d}$ and $F_{(m)}$ converges to zero. Also, if $\mathcal{H}_{0}$ is
not true, this expectation converges to a positive value. The following Lemma indicates
the effect of increasing the value of $a$ and sample size $n$ on the expectation of the posterior
distribution of $d_{AD}(P_{d},F_{(m)})$.

\begin{lemma}\label{Exp of posterior}
Let $P_{d}\sim DP(a+n,H_{d})$ and $H_{d}$ be as given by \eqref{pos base measure} with $H=F_{(m)}$,
the cdf of the chi-square distribution with $m$ degrees of freedom. We have
\begin{enumerate}
\item[$i.$]
$d_{AD}(P_{d}, F_{(m)})\xrightarrow{a.s.} 0$
as $a\rightarrow\infty$ (for fixed $n$).
\item[$ii.$] If $\mathcal{H}_{0}$ is true, then,
$d_{AD}(P_{d}, F_{(m)})\xrightarrow{a.s.} 0$
as $n\rightarrow\infty$ (for fixed $a$).
\item[$iii.$] If $\mathcal{H}_{0}$ is not true, then,
$\liminf d_{AD}(P_{d}, F_{(m)})\displaystyle{\overset{a.s.}{>}}0$
as $n\rightarrow\infty$ (for fixed $a$).
\item[$iv.$] If $\mathcal{H}_{0}$ is not true, then,
$\liminf E_{P_{d}}\left(d_{AD}(P_{d}, F_{(m)})\right)\displaystyle{\overset{a.s.}{>}}0$
as $n\rightarrow\infty$ (for fixed $a$).
\end{enumerate}
\end{lemma}
\begin{proof}
The proof is given in Appendix C.
\end{proof}

\noindent Part (i) of Lemma \ref{Exp of posterior} confirms that for a too large value of $a$ (relative to $n$), we
may accept $\mathcal{H}_{0}$ even when
$\mathcal{H}_{0}$ is not true. More discussion about this point is given in Section 7.
\section{Computational algorithm}

The steps for implementing the proposed test is detailed in this section. Note that, since no closed forms of the densities of $D=d_{AD}(P_{N},F_{(m)})$ and  $D_{d}=d_{AD}(P_{d_{N}},F_{(m)})$ are available, simulation is used to approximate  relative belief ratio. The following gives a computational algorithm to test $\mathcal{H}_{0}$. This algorithm is a revised version of Algorithm B of Al-Labadi and Evans \cite{Al-Labadi2018}.

\noindent\textbf{Algorithm B }\textit{(}\emph{Relative belief algorithm for
testing multivariate normality}\textit{):}

\smallskip
\noindent1. Use Algorithm A to generate (approximately) a $P$  from
$DP(a,F_{(m)})$. \textbf{\smallskip}

\noindent2. Compute $d_{AD}(P_{N},F_{(m)})$ as in Lemma \ref{discrete AD}.\textbf{\smallskip}

\noindent3. Repeat steps (1)-(2) to obtain a sample of $r_{1}$ values from the
prior of $D$.\textbf{\smallskip}

\noindent4. Use Algorithm A to generate (approximately) a $P_d$ from
$DP(a+n,H_{d})$.\textbf{\smallskip}

\noindent5. Compute $d_{AD}(P_{d_{N}},F_{(m)})$.\textbf{\smallskip}

\noindent6. Repeat steps (4)-(5) to obtain a sample of $r_{2}$ values  of $D_{d}$.\textbf{\smallskip}

\noindent7. Let $M$ be a positive number. Let $\hat{F}_{D}$ denote the
empirical cdf of $D$ based on the prior sample in step (3) and for $i=0,\ldots,M,$
let $\hat{d}_{i/M}$ be the estimate of $d_{i/M},$ the $(i/M)$-th prior
quantile of $D.$ Here $\hat{d}_{0}=0$, and $\hat{d}_{1}$ is the largest value
of the prior sample in step (3). Let $\hat{F}_{D}(\cdot\,|\,d)$ denote the empirical cdf of $D_{d}$ based
on the posterior sample in step (6). For $q\in\lbrack\hat{d}_{i/M},\hat
{d}_{(i+1)/M})$, estimate $RB_{D}(q\,|\,d)={\pi_D(q|d)}/{\pi_D(q)}$ by
\begin{equation}
\widehat{RB}_{D}(q\,|\,d)=M\{\hat{F}_{D}(\hat{d}_{(i+1)/M}\,|\,d)-\hat{F}%
_{D}(\hat{d}_{i/M}\,|\,d)\}, \label{rbest}%
\end{equation}
the ratio of the estimates of the posterior and prior contents of $[\hat
{d}_{i/M},\hat{d}_{(i+1)/M}).$ It follows that, we estimate $RB_{D}(0\,|\,d)={\pi_D(0|d)}/{\pi_D(0)}$
 by
$\widehat{RB}_{D}(0\,|\,d)=$ $M\widehat{F}_{D}(\hat{d}_{p_{0}}\,|\,d)$ where
$p_{0}=i_{0}/M$ and $i_{0}$ is chosen so that $i_{0}/M$ is not too small
(typically $i_{0}/M\approx0.05)$.\textbf{\smallskip}

\noindent8. Estimate the strength $DP_{D}(RB_{D}(q\,|\,d)\leq RB_{D}%
(0\,|\,d)\,|\,d)$ by the finite sum
\begin{equation}
\sum_{\{i\geq i_{0}:\widehat{RB}_{D}(\hat{d}_{i/M}\,|\,d)\leq\widehat{RB}%
_{D}(0\,|\,d)\}}(\hat{F}_{D}(\hat{d}_{(i+1)/M}\,|\,d)-\hat{F}_{D}(\hat
{d}_{i/M}\,|\,d)). \label{strest}%
\end{equation}

\noindent For fixed $M,$ as $r_{1}\rightarrow\infty,r_{2}\rightarrow\infty,$
then $\hat{d}_{i/M}$ converges almost surely to $d_{i/M}$ and (\ref{rbest})
and (\ref{strest}) converge almost surely to $RB_{D}(q\,|\,d)$ and
$DP_{D}(RB_{D}(q\,|\,d)\leq RB_{D}(0\,|\,d)\,|\,d)$, respectively.

As recommended in the next section, one should try different values of $a$ to make sure the right conclusion has been obtained. The consistency of the proposed test is achieved by Proposition 6 in Al-Labadi and Evans \cite{Al-Labadi2018}.
\section{Numerical results}
\subsection{Simulation Study}
In this subsection, the performance of the proposed test is illustrated through several examples. Hereafter, let $\mathbf{c}_{m}$ be the $m$-dimensional column vectors of $c$'s, $I_{m}$ be the $m\times m$ identity matrix, $A_{m}$ be an $m\times m$ matrix with $1$'s on the main diagonal and $0.1$'s elsewhere, and $B_{m}$ be an $m\times m$ matrix with $0.25$'s on the main diagonal and $0.2$'s elsewhere. The following notations have been used: $t_{r}$ for the $t$-Student distribution with $r$ degrees of freedom, $\chi^2_{r}$ for the chi-square distribution with $r$ degrees of freedom, $E(\lambda)$ for the exponential distribution with rate $\lambda$, $C(a,b)$ for the Cauchy distribution with location parameter $a$ and scale parameter $b$, $P_{VII}(1,1,r)$ for the Pearson type $VII$ (aka t-Student) distribution with location parameter 1, scale parameter 1 and $r$ degrees of freedom, $t_{r}(\mathbf{0}_{m},I_{m})$ for the $m$-variate $t$-student distribution with location parameter $\mathbf{0}_{m}$, scale parameter $I_{m}$ and $r$ degrees of freedom,  $LN_{m}\left(\mathbf{0}_{m},B_{m}\right)$ for the $m$-variate lognormal distribution with mean vector $\mathbf{0}_{m}$ and covariance matrix $B_{m}$, $\mathcal{S}^{m}(Q)$ for the $m$-variate spherical distribution with distribution $Q$ for radii, and $NMIX$ for the mixture distribution $0.9N_{m}(\mathbf{5}_{m},A_{m})+.1N_{m}(-\mathbf{5}_{m},A_{m})$. Also, for a univariate distribution $Q$, $(Q)^{m}$ denotes an $m$-variate distribution with $m$ identical and independent marginals $Q$, while, $Q_1\otimes\cdots\otimes Q_m$ denotes an $m$-variate distribution with $m$ independent marginals $Q_1,\ldots,Q_m$. In all cases, we generate a sample of size $50$ from an exact distribution and then we set $N=500$, $r_{1} = r_{2} = 1000$ and $M = 20$ in Algorithm B to compute the $RB$ and the corresponding strength. In addition, to study the sensitivity of the approach to the choice of the chosen concentration parameter, various values of $a$ are considered. The results are reported  in Table \ref{T exm4}. The $RB$ of the proposed test is compared to the Bayes factor $BF^{\scalebox{.6}{TM}}$ due to Tokdar and Martin \cite{Tokdar}. Recall that, similar to $RB$, $BF^{\scalebox{.6}{TM}}>1$  supports  $\mathcal{H}_{0}$, $BF^{\scalebox{.6}{TM}}<1$ supports $\mathcal{H}_{1}$ and $BF^{\scalebox{.6}{TM}}=1$ means no
evidence either way. To compute $BF^{\scalebox{.6}{TM}}$, we used the code provided by the same authors and posted at \url{http://www2.stat.duke.edu/~st118/Software/}. From Table \ref{T exm4}, it is seen that for the cases where $\mathcal{H}_0$ is true, the value of $RB$ is greater than $1$ and the strength of this evidence is $1$. This shows a strong evidence to accept $\mathcal{H}_0$. For instance, when $ N_{2}\left(\mathbf{0}_{2},A_{2}\right)$ and $a=15$ in Table \ref{T exm4}, the values of $RB$ and its relevant strength are $7.36$ and $1$, respectively. The prior and posterior density plots of the Anderson-Darling distance for this example are provided by the left side of part (a) of Figure \ref{density-plot} to show that the posterior density of the Anderson-Darling distance is more concentrated around zero than the prior density of the Anderson-Darling distance, when $\mathcal{H}_0$ is true. Also, the Q-Q plot (see Appendix D) of square ordered Mahalanobis distances related to the generated sample $\mathbf{y}_{2\times n}$ against the square root of the percentage points of the chi-square distribution with $m=2$ degrees of freedom is provided by the right side of part (a) of Figure \ref{density-plot}. This figure shows the close relationship between the distribution of observed points and chi-square distribution with $2$ degrees of freedom for this example. Notice that, $(P_{VII}(1,1,r))^{m}$, for $r\geq10$ is a symmetric distribution and has very similar behavior with an $m$-variate normal distribution, and then $\mathcal{H}_0$ is accepted in all cases of this example. For the cases where $\mathcal{H}_0$ is not true, the value of $RB$ is less than $1$ and its relevant strength is small for $a>1$. For example, when $t_{3}(\mathbf{0}_{2},I_{2})$ and $a=15$, the value of $RB$ and its relevant strength are $0.073$ and $0.085$, respectively. The prior and posterior density plots of the Anderson-Darling distance are also provided by the left side of part (b-d) of Figure \ref{density-plot} for various alternative examples to show that the prior density of the Anderson-Darling distance is more concentrated around zero than the posterior density of the Anderson-Darling distance, when $\mathcal{H}_0$ is not true. Note that, while the proposed test rejects $\mathcal{H}_0$ in all examples with $a=15$, the test of Tokdar and Martin \cite{Tokdar} fails to reject $\mathcal{H}_0$ in some case. See, for example, the results for $\mathcal{S}^{m}(LN(0,0.25))$, $ \mathcal{S}^{m}(\chi^{2}_{5})$ and $ t_{3}(\mathbf{0}_{m},I_{m})$ when $a=15$ and $m>2$.


\begin{table}[h!]
\setlength{\extrarowheight}{.1mm}
\setlength{\tabcolsep}{2.5 mm}
\centering
\caption{Relative belief ratios and strength (str)-s for testing the $m$-variate normality assumption
with various alternatives and choices of $a$.}\label{T exm4}
\scalebox{.75}{
\begin{tabular}{lccccccc}
\hline\noalign{\smallskip}
\multirow{2}[3]{*}{\bfseries True distribution}& \multirow{2}[3]{*}{$a$}&\multicolumn{2}{c}{$m=2$}&\multicolumn{2}{c}{$m=3$}&\multicolumn{2}{c}{$m=4$}\\\cmidrule(lr){3-4}\cmidrule(lr){5-6}\cmidrule(lr){7-8}
&& $RB$(str)&$BF^{\scalebox{.6}{TM}}$& $RB$(str)&$BF^{\scalebox{.6}{TM}}$& $RB$(str)&$BF^{\scalebox{.6}{TM}}$\\
\noalign{\smallskip}\hline\noalign{\smallskip}
 & 1 & 19.38(1) &0.197&$17.12(1)$&$10^{-7}$&$18.88(1)$&$2\times10^{-41}$\\
$ N_{m}\left(\mathbf{0}_{m},A_{m}\right)$
 &5 & $14.52(1)$ &$7.039$&$13.73(1)$&$8\times10^{-6}$&$12.99(1)$&$5\times10^{-8}$\\
  & 10 & $9.38(1)$ &$91.50$&$7.09(1)$&$2\times10^{11}$&$8.41(1)$&$10^{22}$\\
 & 15 & $7.36(1)$&$590.2$&$7.92(1)$&$4\times10^{13}$&$7.88(1)$&$4\times10^{22}$\\\noalign{\smallskip}\hline\noalign{\smallskip}

 & 1 & \scalebox{1}{0.24(0.012)} &\scalebox{1}{$3\times 10^{-14}$}&$0(0)$&$6\times10^{-30}$&$0(0)$&$10^{-24}$\\
 \scalebox{1}{$E(\frac{1}{2})\otimes (C(0,1))^{m-1}$ }
  &\scalebox{1}{5}  & \scalebox{1}{0.04(0.004)} &\scalebox{1}{$8\times 10^{-19}$}&$0(0)$&$10^{-50}$&$0(0)$&$5\times10^{-35}$\\
& \scalebox{1}{10} & \scalebox{1}{0(0)} &\scalebox{1}{$2\times 10^{-20}$}&$0(0)$&$2\times10^{-40}$&$0(0)$&$8\times10^{-28}$\\
  & \scalebox{1}{15} & \scalebox{1}{0(0)} &\scalebox{1}{$2\times 10^{-18}$}&$0(0)$&$4\times10^{-36}$&$0(0)$&$2\times 10^{-6}$\\\noalign{\smallskip}\hline\noalign{\smallskip}

 &  \scalebox{1}{1} & \scalebox{1}{0.76(0.054)} &\scalebox{1}{$2\times 10^{-33}$}&$0.1(0.003)$&$10^{-28}$&0(0)&$7\times10^{-34}$\\
 \scalebox{1}{$N(0,1)\otimes(t_{1})^{m-1}$ }
& \scalebox{1}{5}  & \scalebox{1}{0.70(0.067)} &\scalebox{1}{$3\times 10^{-50}$}&0(0)&$10^{-35}$&0(0)&$9\times10^{-45}$\\
 \scalebox{1}{} & \scalebox{1}{10} & \scalebox{1}{0.68(0.058)} &\scalebox{1}{$3\times 10^{-34}$}&$0(0)$&$3\times10^{-25}$&$0(0)$&$10^{-28}$\\
  & \scalebox{1}{15} & \scalebox{1}{0.60(0.055)} &\scalebox{1}{$2\times 10^{-31}$}&$0(0)$&$6\times10^{-15}$&$0(0)$&$10^{-23}$\\\noalign{\smallskip}\hline\noalign{\smallskip}

&1&$0(0)$&$10^{-9}$&$0(0)$&$10^{-20}$&$0(0)$&$10^{-39}$\\
$(P_{VII}(1,1,1))^{m}$\textsuperscript{$\dagger$}&5&$0(0)$&$2\times10^{-11}$&$0(0)$&$6\times10^{-31}$&$0(0)$&$8\times10^{-45}$\\
&10&$0(0)$&$3\times10^{-12}$&$0(0)$&$7\times10^{-26}$&$0(0)$&$4\times10^{-27}$\\
&15&$0(0)$&$10^{-10}$&$0(0)$&$3\times10^{-19}$&$0(0)$&$8\times10^{-13}$\\\noalign{\smallskip}\hline\noalign{\smallskip}

&1&$19.92(1)$&$7\times10^{-2}$&$19.82(1)$&$8\times10^{-7}$&$19.78(1)$&$10^{-5}$\\
$(P_{VII}(1,1,10))^{m}$&5&$17.68(1)$&$3\times10^{2}$&$16.02(1)$&$5\times10^{5}$&$15.84(1)$&$10^{16}$\\
&10&$12.88(1)$&$3\times10^{4}$&$12.76(1)$&$6\times10^{12}$&$11.06(1)$&$6\times10^{26}$\\
&15&$9.62(1)$&$5\times10^{4}$&$10.36(1)$&$10^{14}$&$9.80(1)$&$7\times10^{28}$\\\noalign{\smallskip}\hline\noalign{\smallskip}

  &  \scalebox{1}{1} & \scalebox{1}{1.76(0.321)} &\scalebox{1}{$4\times 10^{-6}$}&$2.12(0.376)$&$10^{-10}$&$1.92(0.363)$&$10^{-12}$\\
\scalebox{1}{$\mathcal{S}^{m}(LN(0,0.25))$\textsuperscript{$\ddagger$}}
& \scalebox{1}{5}  & \scalebox{1}{0.36(0.002)} &\scalebox{1}{$5\times 10^{-5}$}&$0.54(0.028)$&$4\times10^{4}$&$0.16(0.008)$&$10^{14}$\\
 & \scalebox{1}{10} & \scalebox{1}{0.11(0.001)} &\scalebox{1}{$10^{-3}$}&$0.16(0.001)$&$10^{11}$&$0.04(0.003)$&$3\times10^{23}$\\
  & \scalebox{1}{15} & \scalebox{1}{0.01(0.001)} &\scalebox{1}{$9\times 10^{-3}$}&$0.02(0.001)$&$10^{14}$&$0.06(0.004)$&$6\times10^{21}$\\\noalign{\smallskip}\hline\noalign{\smallskip}

  &\scalebox{1}{1}&\scalebox{1}{2.31(0.341)}&\scalebox{1}{0.408}&2.11(0.393)&$10^{-8}$&1.22(0.308)&$3\times10^{-38}$\\
\scalebox{1}{$ \mathcal{S}^{m}(\chi^{2}_{5})$}&\scalebox{1}{5}&\scalebox{1}{0.86(0.155)}&\scalebox{1}{294.02}&0.78(0.143)&$296.66$&0.34(0.002)&$7\times10^{8}$\\
&\scalebox{1}{10}&\scalebox{1}{0.56(0.028)}&\scalebox{1}{4459.4}&0.22(0.011)&$10^{10}$&0.12(0.003)&$5\times10^{20}$\\
&\scalebox{1}{15}&\scalebox{1}{0.32(0.001)}&\scalebox{1}{31538}&0.1(0.005)&$7\times10^{11}$&0.05(0)&$10^{19}$\\\noalign{\smallskip}\hline\noalign{\smallskip}

&\scalebox{1}{1}&\scalebox{1}{1.82(0.301)}&\scalebox{1}{$3\times 10^{-36}$}&$1.69(0.359)$&$2\times10^{-137}$&$1.92(0.340)$&$0$\\
\scalebox{1}{$ LN_{m}\left(\mathbf{0}_{m},B_{m}\right)$\textsuperscript{$\ast$}}&\scalebox{1}{5}&\scalebox{1}{0.22(0.022)}& \scalebox{1}{$3\times 10^{-64}$}&$0.06(0.003)$&$10^{-158}$&$0.26(0.006)$&$6\times10^{-232}$\\
&\scalebox{1}{10}&\scalebox{1}{0.08(0.004)}& \scalebox{1}{$7\times 10^{-46}$}&$0.02(0.001)$&$2\times10^{-147}$&$0.02(0.001)$&$8\times10^{-56}$\\
&\scalebox{1}{15}&\scalebox{1}{0.06(0.003)}& \scalebox{1}{$2\times 10^{-33}$}&$0.02(0.003)$&$10^{-76}$&$0.01(0.008)$&$3\times10^{-34}$\\\noalign{\smallskip}\hline\noalign{\smallskip}

&\scalebox{1}{1}&\scalebox{1}{1.42(0.328)}& \scalebox{1}{$3\times 10^{-8}$}&$1.63(0.311)$&$10^{-5}$&$1.81(0.337)$&$9\times10^{-24}$\\
\scalebox{1}{$ t_{3}(\mathbf{0}_{m},I_{m})$\textsuperscript{$\ddagger$}}&\scalebox{1}{5}&\scalebox{1}{0.88(0.105)}&\scalebox{1}{$3\times 10^{-5}$}&$0.90(0.124)$&$10^{3}$&$0.82(0.111)$&$10^{14}$\\
&\scalebox{1}{10}&\scalebox{1}{0.81(0.097)}&\scalebox{1}{$10^{-3}$}&$0.70(0.090)$&$3\times10^{8}$&$0.66(0.053)$&$9\times10^{23}$\\
&\scalebox{1}{15}&\scalebox{1}{0.73(0.085)}&\scalebox{1}{$5\times 10^{-3}$}&$0.50(0.021)$&$10^{12}$&$0.43(0.006)$&$2\times10^{27}$\\\noalign{\smallskip}\hline\noalign{\smallskip}

&\scalebox{1}{1}&\scalebox{1}{1.68(0.281)}&\scalebox{1}{$10^{-53}$}&$2.13(0.312)$&$0$&$2.18$&$0$\\
\scalebox{1}{$NMIX$\textsuperscript{$\ddagger$}} &\scalebox{1}{5}&\scalebox{1}{0.12(0.006)}& \scalebox{1}{$9\times 10^{-48}$}&$0.54(0.121)$&$0$&$0.67(0.152)$&$0$\\
&\scalebox{1}{10}&\scalebox{1}{0.06(0.003)}& \scalebox{1}{$10^{-50}$}&$0.23(0.026)$&$0$&$0.34(0.006)$&$0$ \\
&\scalebox{1}{15}&\scalebox{1}{0(0)}& \scalebox{1}{$3\times 10^{-41}$}&$0.09(0.010)$&$0$&$0.02(0.007)$&$0$\\\noalign{\smallskip}\hline\noalign{\smallskip}

\end{tabular}}
\centering
\begin{tablenotes}
      \small
      \item \fontsize{8.5}{8}\selectfont{\textsuperscript{$\dagger$}Required $\mathsf{R}$ package to generate sampled data: PearsonDS.}
      \item \fontsize{8.5}{8}\selectfont{\textsuperscript{$\ddagger$}Required $\mathsf{R}$ package to generate sampled data: distrEllipse.}
      \item \fontsize{8.5}{8}\selectfont{\textsuperscript{$\ast$}Required $\mathsf{R}$ package to generate sampled data: compositions.}
    \end{tablenotes}
\end{table}

\begin{figure}[H]
 \centering
    \subfloat[$N_{2}(\mathbf{0}_{2},A_2)$]{{\includegraphics[width=4cm]{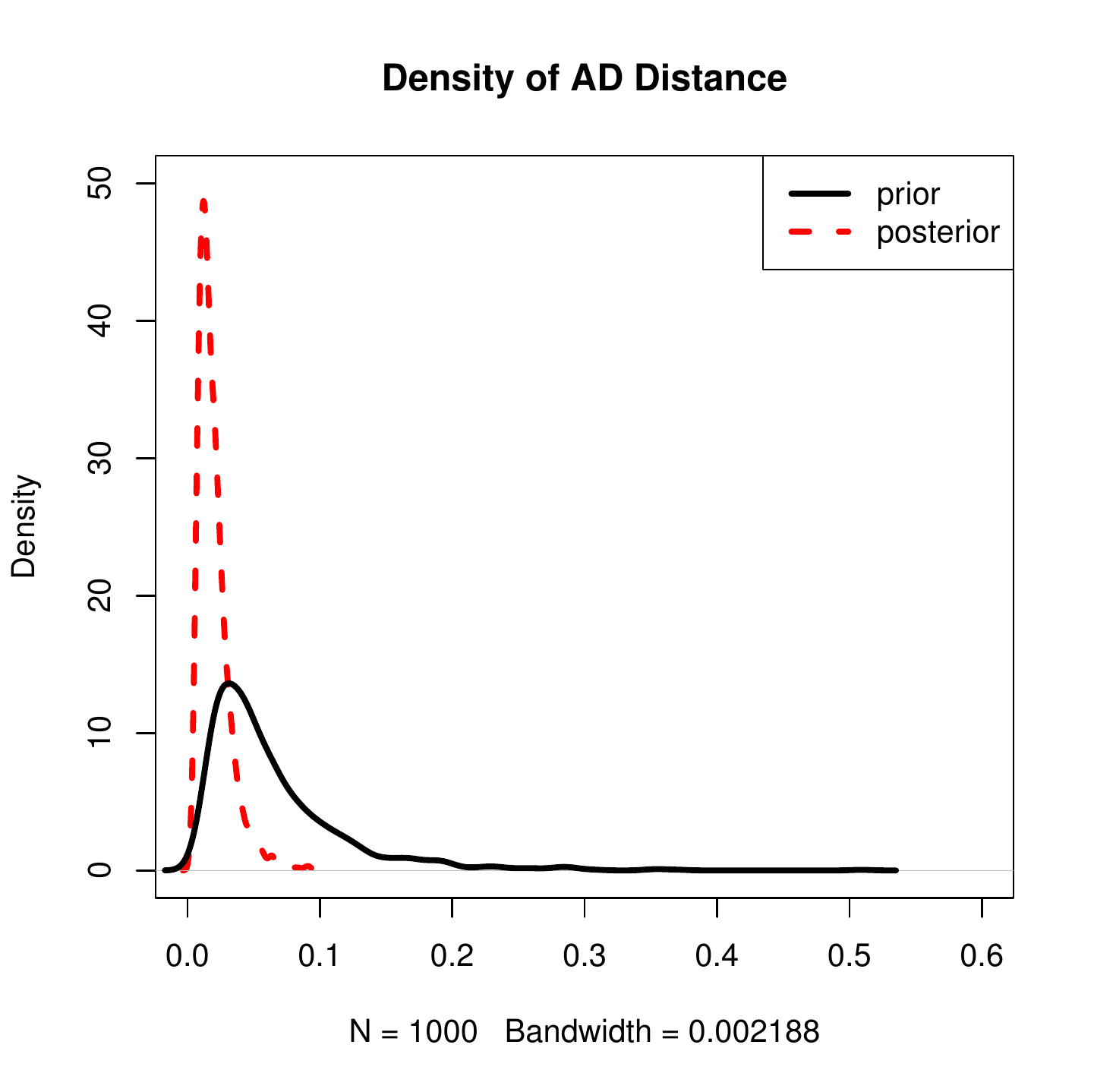} }
    \hspace{.1cm}
    {\includegraphics[width=4cm]{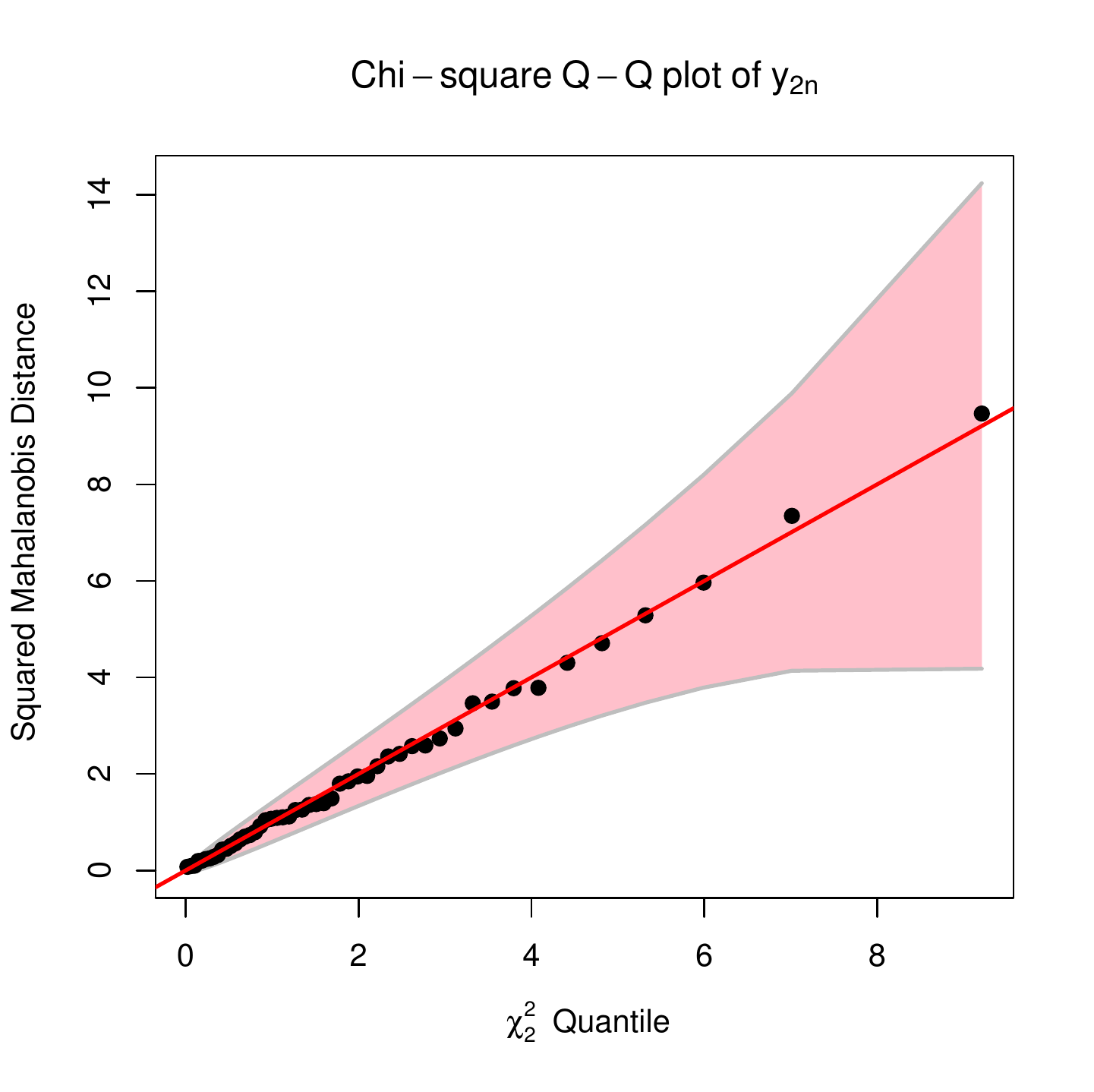} }}%
    \\
    \subfloat[$(P_{VII}(1,1,1))^{2}$]{{\includegraphics[width=4cm]{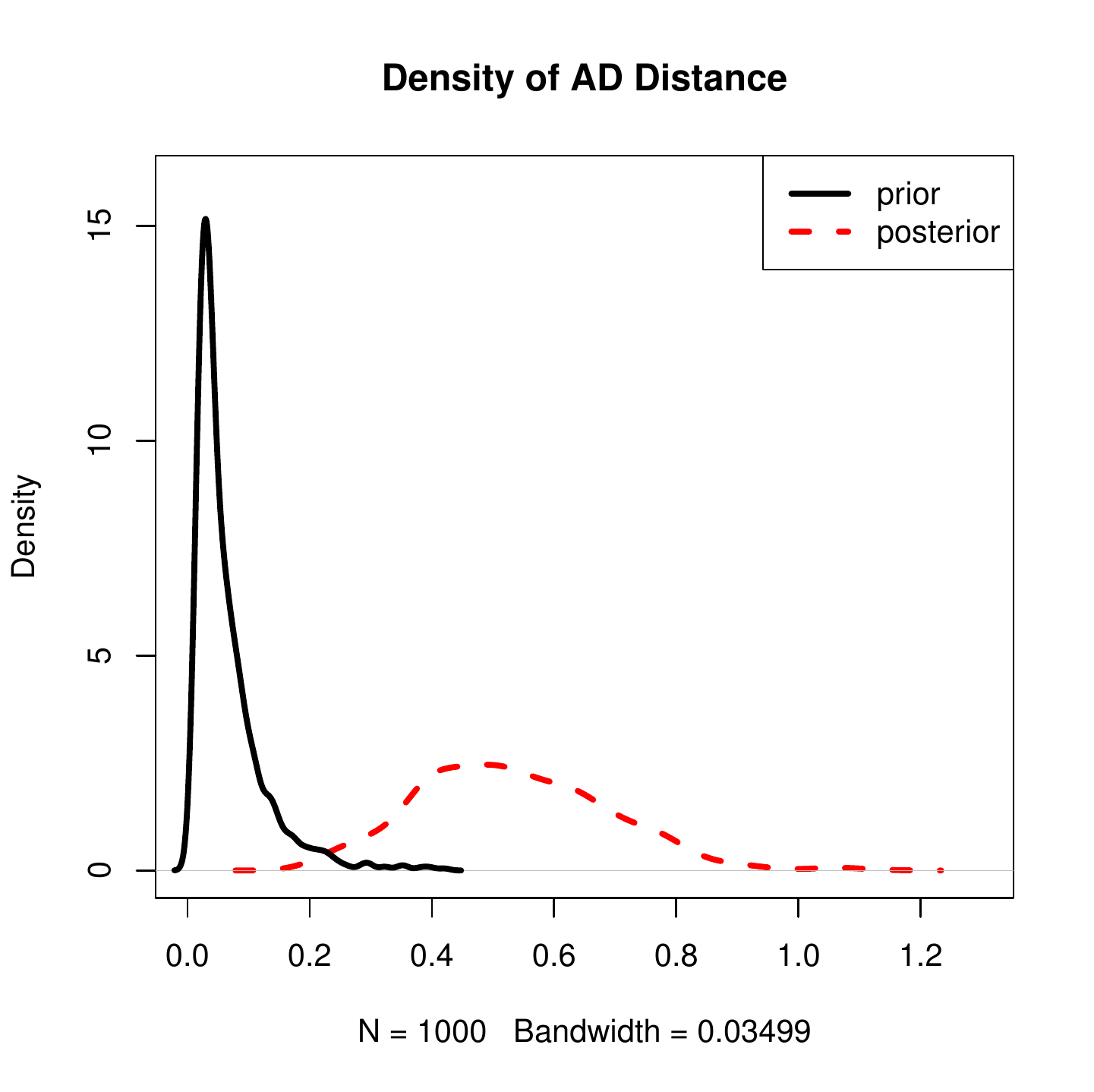} }%
    \hspace{.1cm}
    {\includegraphics[width=4cm]{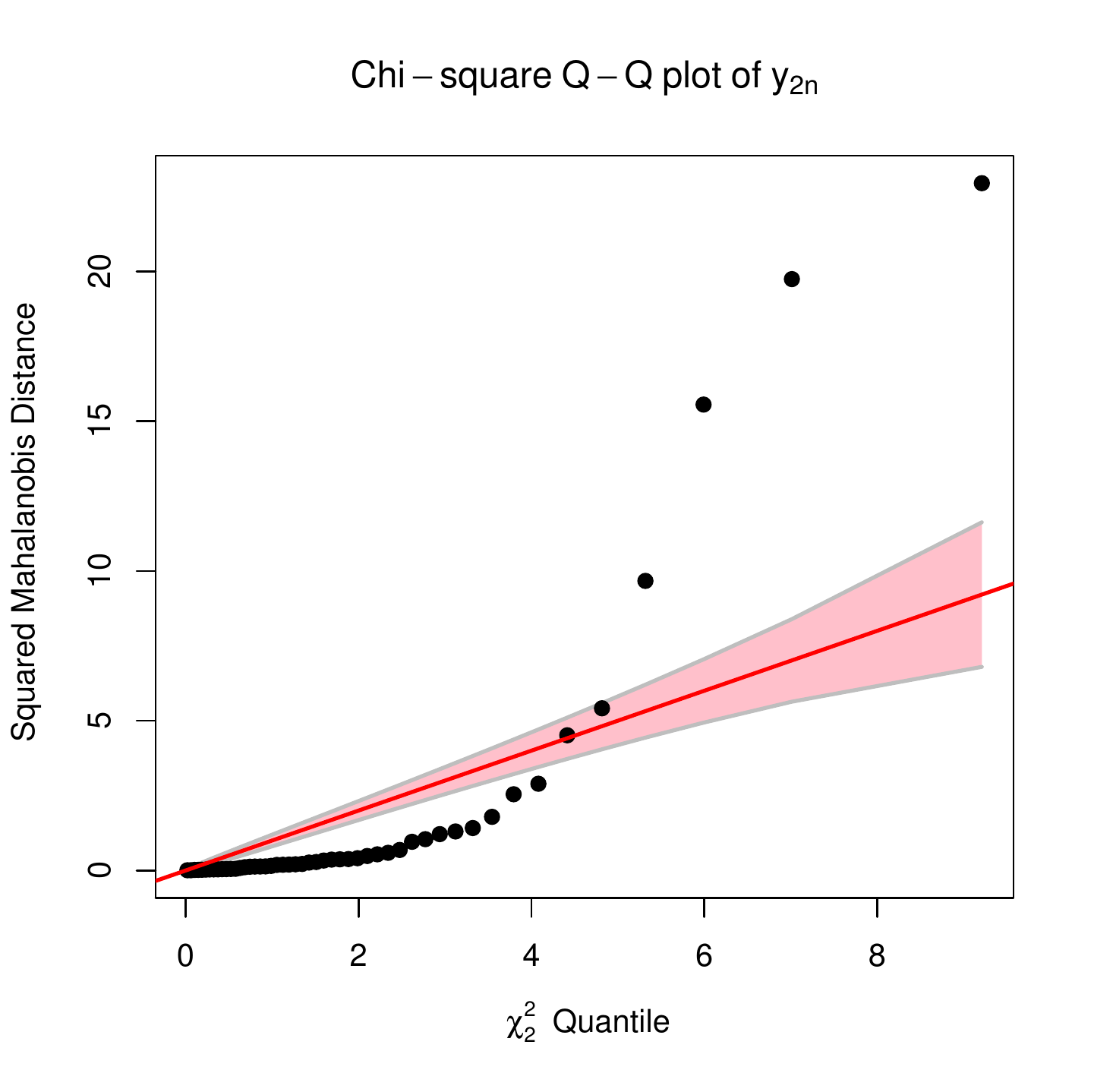} }}
    \\
    \subfloat[$t_{3}(\mathbf{0}_{2},I_{2})$]{{\includegraphics[width=4cm]{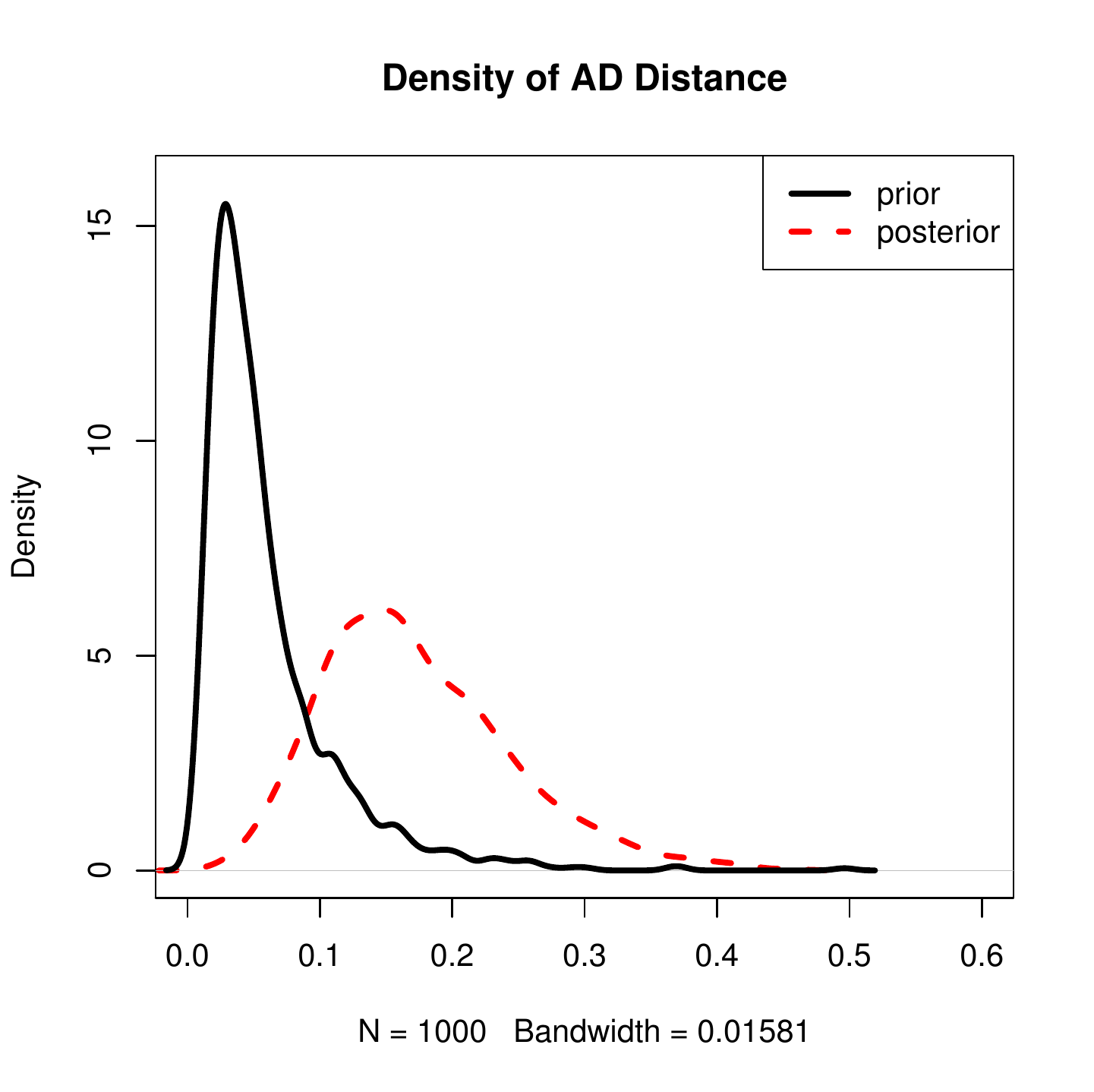} }
    \hspace{.1cm}
   {\includegraphics[width=4cm]{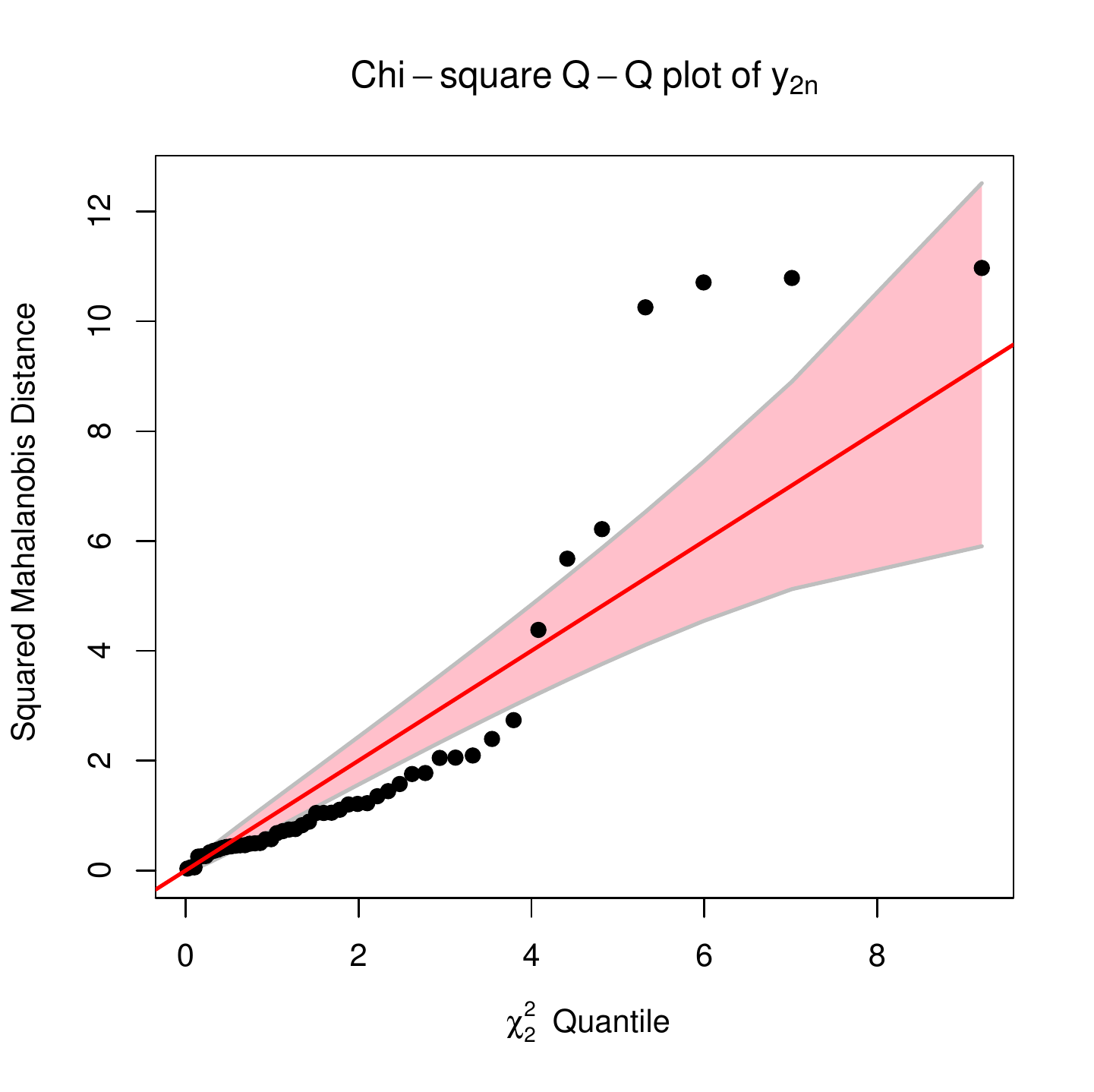} }}
    \\
    \subfloat[$N(0,1)\otimes t_{1}$]{{\includegraphics[width=4cm]{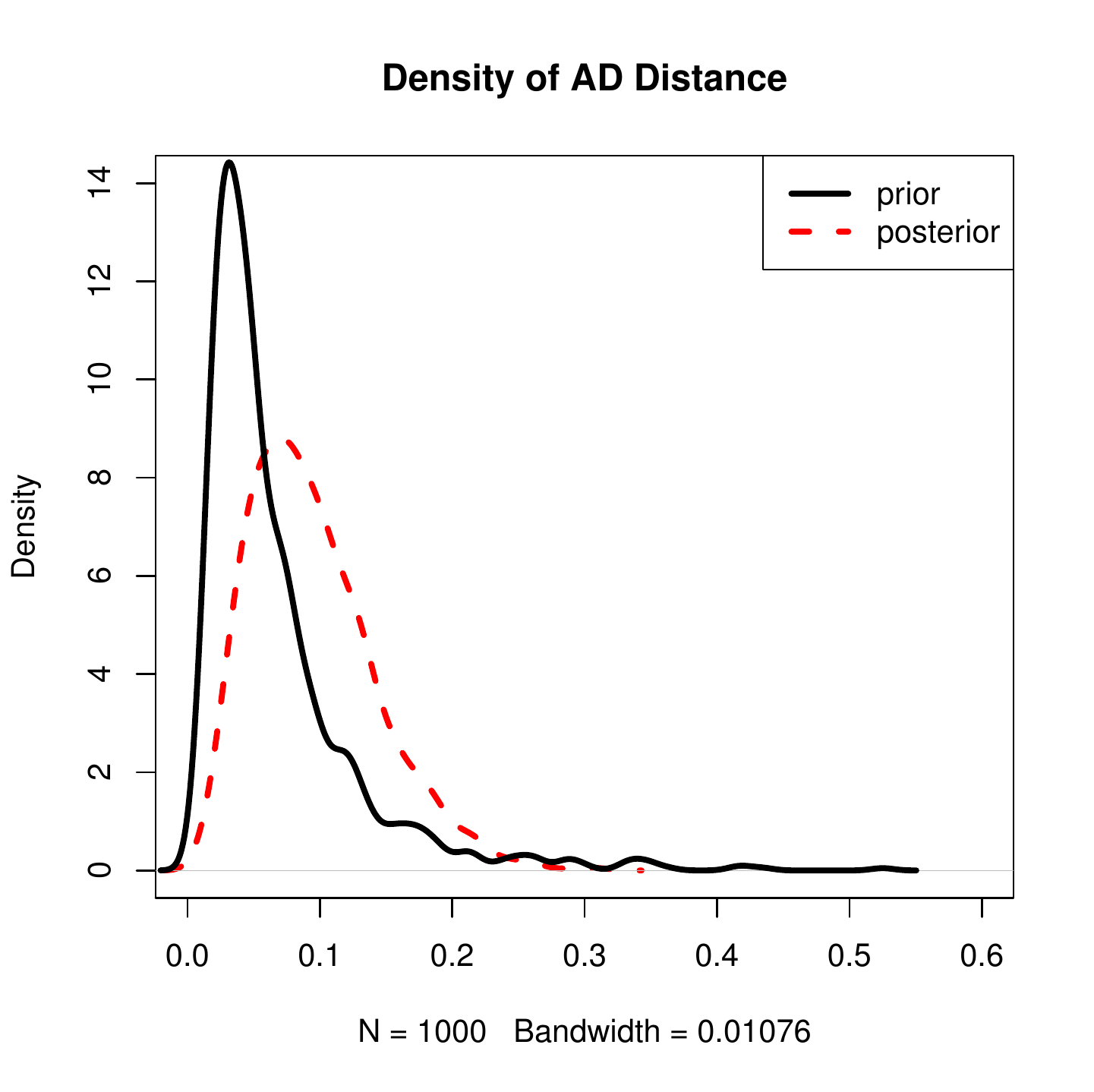} }
    \hspace{.1cm}
   {\includegraphics[width=4cm]{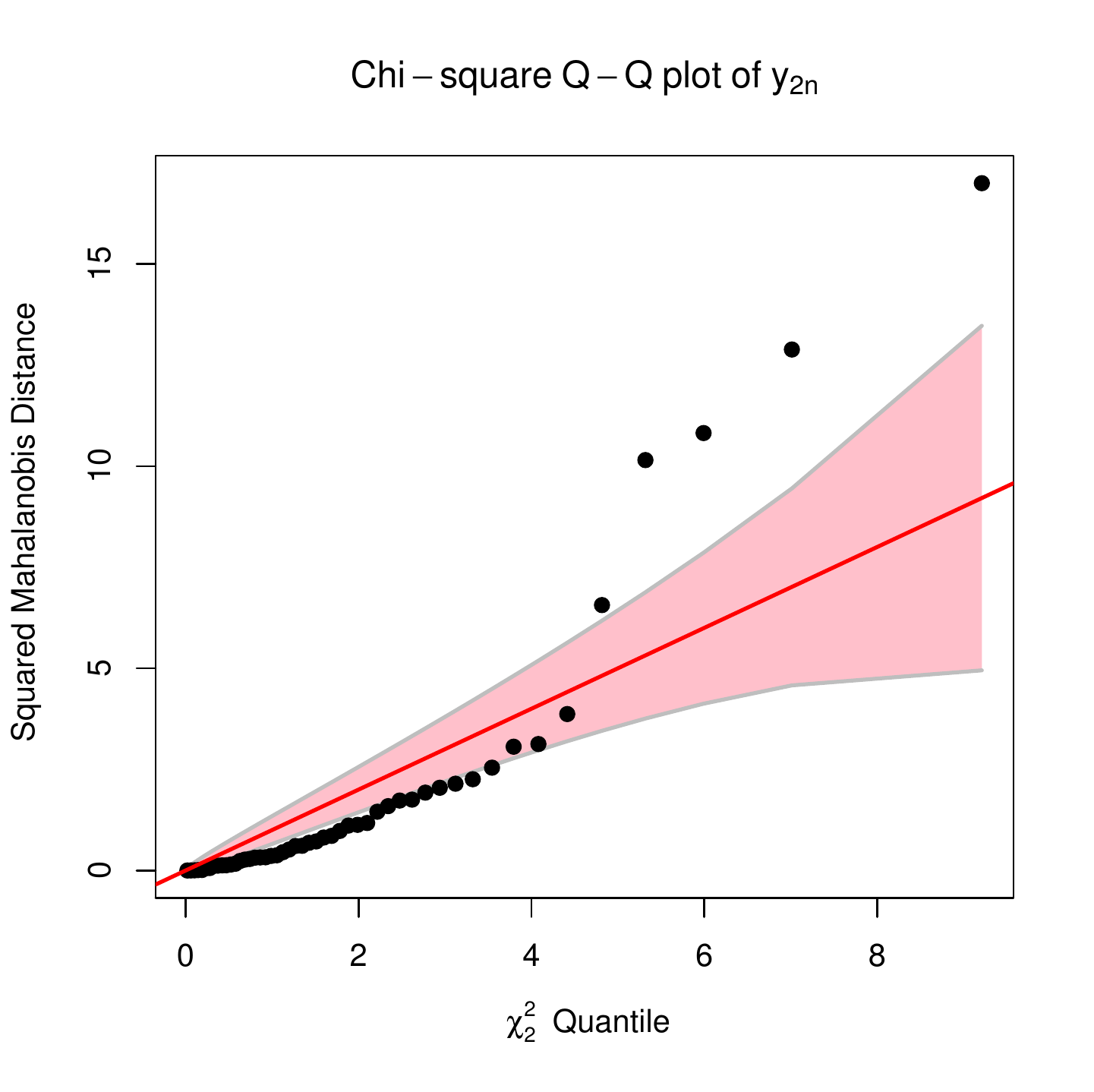} }}
    \caption{Left: The solid line represents the prior density of the Anderson-Darling distance and the dashed line represents the posterior density of the Anderson-Darling distance for $a=15$. Right: Chi-square Q-Q plot for the simulated sample with $95\%$ envelope. }\label{density-plot}%
\end{figure}
We end this subsection by investigating  the effect of the choice of $H$ in the proposed test. We consider the results of the proposed test for a sample of size $50$ from true distributions $N_{2}(\mathbf{0}_{2},A_2)$, $t_{3}(\mathbf{0}_{2},I_{2})$ and $N(0,1)\otimes t_{1}$ for different choices of $H$ in Table \ref{prior-data}. Clearly, when $\mathcal{H}_0$ is not true, the results are correct for only $H=\chi^2_{2}$; otherwise, they are incorrect.
\begin{table}[h!]
\centering
\setlength{\tabcolsep}{4 mm}
\caption{$RB$(Strength) of a sample of size 50 when there is prior-data conflict (a tiny overlap between the effective support regions).}\label{prior-data}
\begin{tabular}{lcccc}
\toprule\noalign{\smallskip}
True distribution&$H$ &$a=5$ &$a=10$&$a=15$  \\
\noalign{\smallskip}\hline\noalign{\smallskip}
$N_{2}(\mathbf{0}_{2},A_2)$&\scalebox{1}{$\chi^{2}_{2}$} &\scalebox{1}{$14.52(1)$} &\scalebox{1}{$9.38(1)$} &$7.36(1)$ \\
&\scalebox{1}{$C(0,1)$} &\scalebox{1}{$20(1)$} &\scalebox{1}{$20(1)$}&$20(1)$ \\
&\scalebox{1}{$N(0,1)$} &\scalebox{1}{$20(1)$} &\scalebox{1}{$20(1)$}&$20(1)$ \\
\noalign{\smallskip}\hline\noalign{\smallskip}
$t_{3}(\mathbf{0}_{2},I_2)$&$\chi^{2}_{2}$&$0.88(0.105)$&$0.81(0.097)$&$0.73(0.087)$\\
&$C(0,1)$&$20(1)$&$20(1)$&$19.38(1)$\\
&$N(0,1)$&$20(1)$&$20(1)$&$20(1)$\\
\noalign{\smallskip}\hline\noalign{\smallskip}
$N(0,1)\otimes t_{1}$&$\chi^{2}_{2}$&$0.76(0.054)$&$0.70(0.067)$&$0.68(0.058)$\\
&$C(0,1)$&$20(1)$&$20(1)$&$20(1)$\\
&$N(0,1)$&$20(1)$&$20(1)$&$20(1)$\\
\bottomrule
\end{tabular}
\end{table}

\subsection{Application}
In this subsection, we look at the performance of the methodology by using a real data set. For comparison purposes, the results of the Doornik-Hanson (DH) test are also presented \cite{Doornik}. The $\mathsf{R}$ package \textbf{asbio} is used to compute p-values of the DH test.

We consider the data set of the National track records for women which reprints women's athletic records for 55 countries with seven variables (Atkinson et al. \cite{Atkinson}). The variables are 100, 200, 400 meters in seconds, 800, 1500, 3000 meters in minutes and the marathon. The problem is to assess the seven-variate normality assumption for this data set. Atkinson et al. \cite{Atkinson} provided the Q-Q plot (see the right side of Figure \ref{QQ plot}) of squared ordered Mahalanobis distances against the square root of the percentage points of the chi-square distribution with seven degrees of freedom and then suggested the non-normality of this data set. The DH's p-value for this data is $4.042382\times 10^{-68}$, which shows strong evidence to reject the multivariate normality assumption. The results of the proposed test are given in Table \ref{T exm6}, which show a strong evidence to reject the assumption of multivariate normality .

\begin{table}[ht]
\centering
\setlength{\tabcolsep}{.7 mm}
\caption{Relative belief ratios and strength (str)-s for testing the seven-variate normality assumption
of national track records for women with various choices of $a$.}\label{T exm6}
\begin{small}
\begin{tabular}{ccccccc}
\toprule
& \multicolumn{6}{c}{$a$} \\ \cmidrule{2-7} \bfseries{Real data}& 1& 5&6& 8 &10&15\\ \hline
$RB$(str) &7.480(0.626) &1.140(0.405) &0.700(0.133) & 0.480(0.063)& 0.240(0.012) &0.120(0.006) \\
\bottomrule
\end{tabular}
\end{small}
\end{table}
\begin{figure}[H]
\centering
    \subfloat{{\includegraphics[width=5cm]{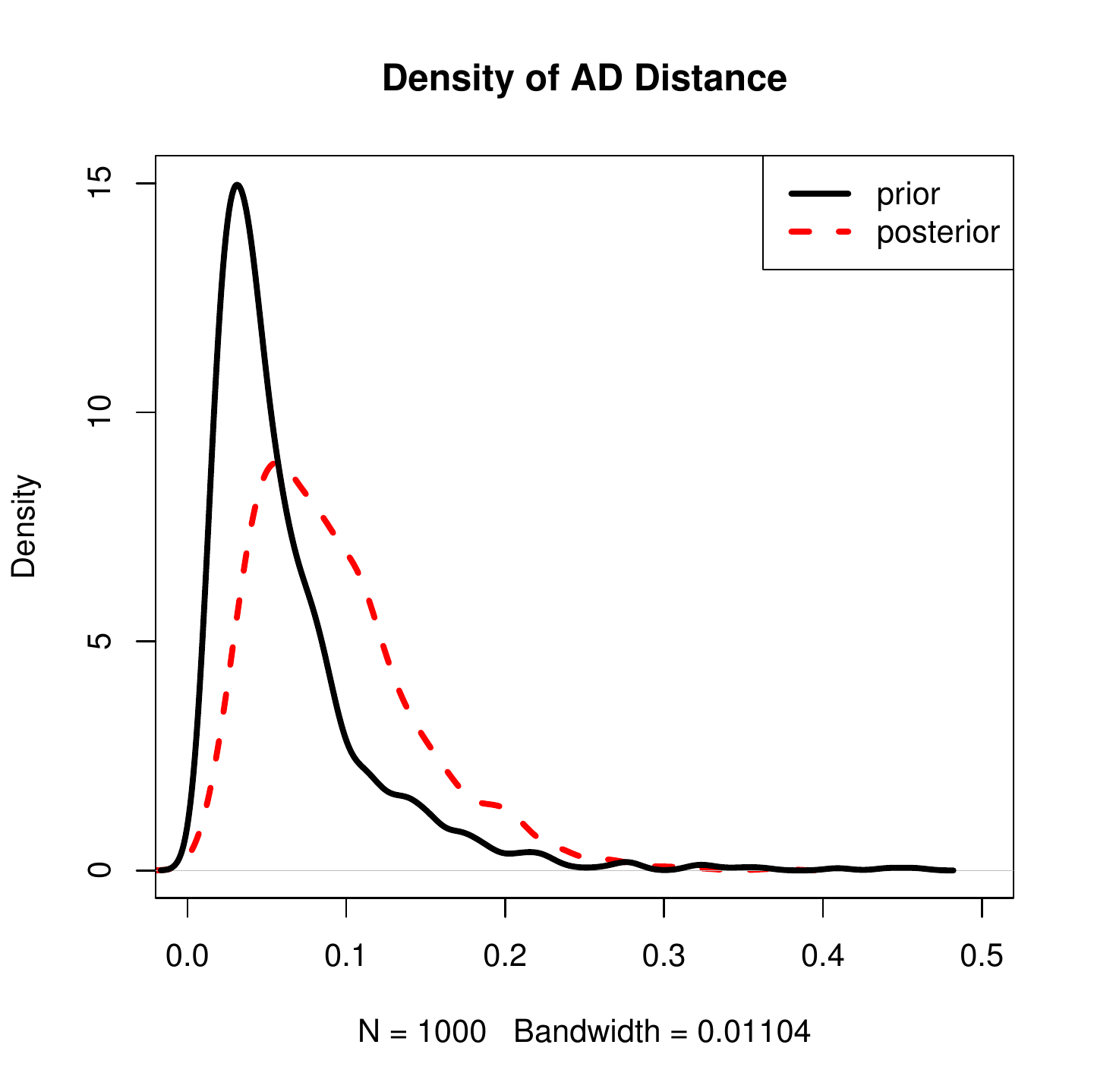} }
    \hspace{.1cm}
    {\includegraphics[width=5cm,height=5.2cm]{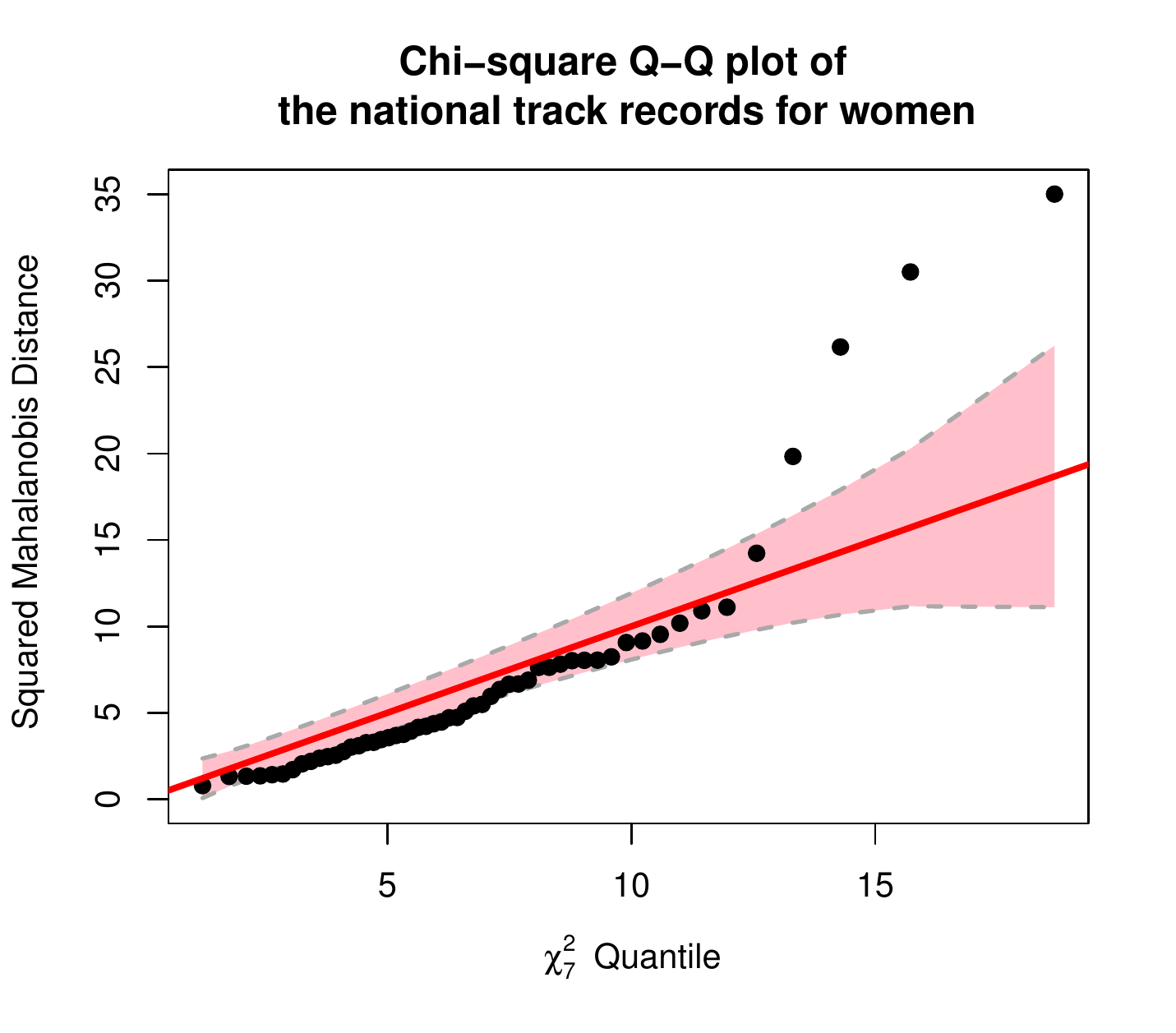} }}%
    \caption{Left: The solid line represents the prior density of the Anderson-Darling distance and the dashed line represents the posterior density of the Anderson-Darling distance for $a=15$. Right: Chi-square Q-Q plot of the real data set with $95\%$ envelope.}%
    \label{QQ plot}%
\end{figure}

\section{Conclusion}
A novel Bayesian nonparametric test for assessing multivariate normality has been proposed. The suggested test is developed by using Mahalanobis distance as a good technique to convert the $m$-variate problem into the univariate problem. The Dirichlet process has been considered as a prior on the distribution of the Mahalanobis distance. Then, the concentration of the distribution of the distance between posterior process and chi-square distribution with $m$ degrees of freedom is compared to the concentration of the distribution of the distance between prior process and chi-square distribution with $m$ degrees of freedom via relative belief ratio.  The distance between the Dirichlet process and the chi-square distribution is developed based on the Anderson-Darling distance. Several theoretical results including consistency have been discussed for the proposed test. The test is illustrated through various simulated examples and a real data set.

\begin{appendices}
\section{Proof of Lemma \ref{discrete AD}}
Consider $P_{N}(x)$ as
\begin{small}
$$
P_{N}(x)=
\begin{cases}
0& x<Y_{(1)}\\
P_{N}(Y_{(i)})& Y_{(i)}\leq x <Y_{(i+1)},\, (i=1,\ldots ,N-1).\\
1& x\geq Y_{(N)}
\end{cases}
$$
\end{small}
Let $g(x)=\frac{dG(x)}{dx}$, then

\begin{small}
\begin{align*}
d_{AD}(P_{N},G)&=\int_{-\infty}^{\infty}\frac{\left(P_{N}(x)-G(x)\right)^{2}}{G(x)\left(1-G(x)\right)}g(x)\, dx\\
&=\int_{-\infty}^{Y_{(1)}}\frac{G(x)^{2}}{G(x)\left(1-G(x)\right)}g(x)\, dx
+\int_{Y_{(N)}}^{\infty}\frac{\left( 1-G(x)\right)^{2}}{G(x)\left(1-G(x)\right)}g(x)\, dx\\
&+\sum_{i=1}^{N-1}\int_{Y_{(i)}}^{Y_{(i+1)}}\frac{\left(P_{N}(Y_{(i)})-G(x)\right)^{2}}{G(x)\left(1-G(x)\right)}g(x)\, dx.\\
\end{align*}
\end{small}
Substituting $y=G(x)$, $G(Y_{(i)})=U_{(i)}$, $G(-\infty)=0$ and $G(\infty)=1$, gives

\begin{small}
\begin{align*}
d_{AD}(P_{N},G)&=\sum_{i=1}^{N-1}\int_{U_{(i)}}^{U_{(i+1)}}\frac{\left(P_{N}(Y_{(i)})-y\right)^{2}}{y\left(1-y\right)}\, dy+\int_{0}^{U_{(1)}}\frac{y}{1-y}\, dy+\int_{U_{(N)}}^{1}\frac{1-y}{y}\, dy\\
&=\sum_{i=1}^{N-1} \bigg[P_{N}^{2}(Y_{(i)})\log(y)-\Big[\left(P_{N}(Y_{(i)})-1\right)^{2}\log(1-y)\Big]-y\bigg]_{U_{(i)}}^{U_{(i+1)}}\\
&+\bigg[-y-\log(1-y)\bigg]_{0}^{U_{(1)}}+\bigg[\log(y)-y\bigg]_{U_{(N)}}^{1}\\
&=I_{1}+I_{2}+I_{3}.
\end{align*}
\end{small}
Note that,

\begin{small}
\begin{align*}
I_{1}&=-\sum_{i=1}^{N-1}\left(U_{(i+1)}-U_{(i)}\right)-\sum_{i=1}^{N-1}\left(P_{N}(Y_{(i)})-1\right)^{2}
\left(\log(1-U_{(i+1)})-\log(1-U_{(i)})\right)\nonumber\\
&+\sum_{i=1}^{N-1}P_{N}^{2}(Y_{(i)})\left(\log(U_{(i+1)})-\log(U_{(i)})\right)\nonumber\\
&=\sum_{i=1}^{N-1}P_{N}^{2}(Y_{(i)})\log\frac{U_{(i+1)}\left(1-U_{(i)}\right)}{U_{(i)}\left(1-U_{(i+1)}\right)}+\sum_{i=1}^{N-1}\left(2P_{N}(Y_{(i)})-1\right)\log\frac{1-U_{(i+1)}}{1-U_{(i)}}
\nonumber\\
&-\left(U_{(N)}-U_{(1)}\right).
\end{align*}
\end{small}
\noindent Also, $I_{2}=-U_{(1)}-\log\left(1-U_{(1)}\right)$ and $I_{3}=-1-\log U_{(N)}+U_{(N)}.$ Therefore,
adding $I_{1}$, $I_{2}$ and $I_{3}$, gives

\begin{small}
\begin{align}\label{app of AD-2}
d_{AD}(P_{N},G)&=\sum_{i=1}^{N-1}P_{N}^{2}(Y_{(i)})\log\frac{U_{(i+1)}\left(1-U_{(i)}\right)}{U_{(i)}\left(1-U_{(i+1)}\right)}+\sum_{i=1}^{N-1}\left(2P_{N}(Y_{(i)})-1\right)\log\frac{1-U_{(i+1)}}{1-U_{(i)}}
\nonumber\nonumber\\
&-1-\log\left(U_{(N)}(1-U_{(1)})\right).
\end{align}
\end{small}
The proof is completed by substituting $P_{N}(Y_{(i)})=\sum_{j=1}^{i}J^{\prime}_{j}$, $P^{2}_{N}(Y_{(i)})=
\sum_{j=1}^{i}J^{\prime^{2}}_{j}+2\sum_{j=1}^{i-1}\sum_{k=j+1}^{i}J^{\prime}_{k}J^{\prime}_{j}$
and $U_{(i)}=G(Y_{(i)})$ in terms on the right-hand side of \eqref{app of AD-2}.

\section{Proof of Lemma \ref{Exp-Var of AD}}
To prove (i), note that, from the property of the Dirichlet process, for any $t\in \mathbb{R}$, $E_{P}( P(t)-$ $H(t))^{2}=\frac{H(t)( 1-H(t))}{a+1}$. Then
\begin{small}
\begin{equation*}
E_{P}\left( d_{AD}(P, H)\right)=\int_{-\infty}^{\infty}\frac{E_{P}\left( P(t)-H(t)\right)^{2}}{H(t)\left( 1-H(t)\right)}\, dH(t)=\frac{1}{a+1}.
\end{equation*}
\end{small}
To prove (ii), it is enough to compute $E_{P}\left( d_{AD}(P, H)\right)^{2}$. According to the
Corollary \ref{corollary}, we consider $H$
to be the cdf of the Uniform distribution on [0, 1]. Then
\begin{small}
\begin{align*}
E_{P}\left( d_{AD}(P, H)\right)^{2}&=E_{P}\left(\int_{-\infty}^{\infty}\frac{\left( P(t)-H(t)\right)^{2}}{H(t)\left( 1-H(t)\right)}\, dH(t)\right)^{2}\\
&=E_{P}\left(\int_{0}^{1}\frac{\left( P(t)-t\right)^{2}}{t(1-t)}\, dt
\int_{0}^{1}\frac{\left( P(s)-s\right)^{2}}{s(1-s)}\, ds\right)\\
&=E_{P}\Bigg(\int_{0}^{1}\int_{0}^{t}\frac{\left( P(t)-t\right)^{2}\left( P(s)-s\right)^{2}}{t(1-t)s(1-s)}\, ds\,dt\\
&+\int_{0}^{1}\int_{0}^{s}\frac{\left( P(t)-t\right)^{2}\left( P(s)-s\right)^{2}}{t(1-t)s(1-s)}\, dt\,ds\Bigg)\\
&=2E_{P}\Bigg(\int_{0}^{1}\int_{0}^{t}\frac{\left( P(t)-t\right)^{2}\left( P(s)-s\right)^{2}}{t(1-t)s(1-s)}\, ds\,dt\Bigg)\\
&=2\int_{0}^{1}\int_{0}^{t}\frac{E_{P}\big\lbrace\left( P(s)+P\left((s,t]\right)-t\right)^{2}\left( P(s)-s\right)^{2}\big\rbrace}{t(1-t)s(1-s)}\, ds\,dt.\\
\end{align*}
\end{small}
Note that, from the property of the Dirichlet process, for any $s<t$ and $i,j\in\mathbb{N}$,
$E_{P}(P^{i}(s)P^{j}((s,t])) =\frac{\Gamma(a)}{\Gamma(a+i+j)}\frac{\Gamma(as+i)}{\Gamma(as)}
\frac{\Gamma\left(a(t-s)+j\right)}{\Gamma(a(t-s))}$ and $E_{P}\left(P^{i}(s)\right)=\prod_{k=0}^{i-1}\frac{as+k}{a+k}$. Then

\begin{small}
\begin{align*}
E_{P}\left( d_{AD}(P, H)\right)^{2}&=\int_{0}^{1}\int_{0}^{t}\frac{1}{ts(1-t)(1-s)}\bigg\lbrace
\frac{2as(a(t-s)+1)(as+1)(t-s)}{(a+3)(a+2)(a+1)}\\
&+\frac{4as(as+2)(as+1)(t-s)}{(a+3)(a+2)(a+1)}+\frac{2s(as+3)(as+2)(as+1)}{(a+3)(a+2)(a+1)}\\
&-\frac{4as(2s+t)(as+1)(t-s)}{(a+2)(a+1)}-\frac{4as^{2}(a(t-s)+1)(t-s)}{(a+2)(a+1)}\\
&-\frac{4s(as+2)(as+1)(s+t)}{(a+2)(a+1)}+\frac{2s(s^{2}+4st+t^{2})(as+1)}{a+1}\\
&+\dfrac{4as^{2}(s+2t)(t-s)}{a+1}+\frac{2s^{2}(t-s)(a(t-s)+1)}{a+1}-4s^{2}t(t-s)\\
&-4s^{2}t(s+t)+2s^{2}t^{2}\bigg\rbrace\,ds\,dt.
\end{align*}
\end{small}
After simplification, we get

\begin{small}
\begin{align*}
E_{P}\left( d_{AD}(P, H)\right)^{2}&=\int_{0}^{1}\int_{0}^{t}\frac{2\left((a-6)\left((3t-2)s-t\right)-6\right)}{t(s-1)(a+3)(a+2)(a+1)}\, ds\, dt\\
&=\int_{0}^{1}\frac{2}{t(a+3)(a+2)(a+1)}\Big\lbrace(a-6)(3t-2)t-2i\pi\left((a-6)t-a+3\right)\\
&+2\left(a(t-1)-6t+3\right)\log(t-1)\Big\rbrace\,dt\\
&=\frac{a(2\pi^{2}-15)-6(\pi^{2}-15)}{3(a+3)(a+2)(a+1)},
\end{align*}
\end{small}
for Re$(t)<1$ or $t\not\in\mathbb{R}$, where Re$(t)$ denotes the real part of $t$ and $i$ is the
imaginary unit. Then, the variance of $d_{AD}(P, H)$ is given by

\begin{small}
\begin{align*}
Var_{P}(d_{AD}(P, H))&=E_{P}\left( d_{AD}(P, H)\right)^{2}-E_{P}^{2}\left( d_{AD}(P, H)\right)\\
&=\dfrac{2\left( (\pi^{2}-9)a^{2}+(30-2\pi^{2})a-
3\pi^{2}+36\right)}{3(a+1)^{2}(a+2)(a+3)}.
\end{align*}
\end{small}
Hence, the proof is completed.

\section{Proof of Lemma \ref{Exp of posterior}}
Assume that $d=(d_{M}^{2}(\mathbf{y}_{1}),\ldots,
d_{M}^{2}(\mathbf{y}_{n}))$ is the observed square Mahalanobis distance
from $P$ where $P\sim DP(a,F_{(m)})$. Note that
\begin{small}
\begin{align*}
d_{AD}(P_{d},F_{(m)})&=\int_{-\infty}^{\infty}\frac{\left( P_{d}(t)-F_{(m)}(t)\right)^{2}}{F_{(m)}(t)\left( 1-F_{(m)}(t)\right)}\, dF_{(m)}(t)\nonumber\\
&\leq\left(\displaystyle{\sup_{t\in\mathbb{R}}}| P_{d}(t)-F_{(m)}(t)|\right)^{2}\int_{-\infty}^{\infty}\frac{1}{F_{(m)}(t)\left( 1-F_{(m)}(t)\right)} dF_{(m)}(t)\nonumber\\
&\leq\displaystyle{\sup_{t\in\mathbb{R}}}| P_{d}(t)-F_{(m)}(t)|\int_{-\infty}^{\infty}\frac{1}{F_{(m)}(t)\left( 1-F_{(m)}(t)\right)}dF_{(m)}(t)\\
&\leq\int_{-\infty}^{\infty}\frac{1}{F_{(m)}(t)\left( 1-F_{(m)}(t)\right)}dF_{(m)}(t)\bigg\lbrace\displaystyle{\sup_{t\in\mathbb{R}}}| P_{d}(t)-H_{d}(t)|+\\
&\displaystyle{\sup_{t\in\mathbb{R}}}| H_{d}(t)-F_{(m)}(t)|\bigg\rbrace,
\end{align*}
\end{small}
where the third inequality holds since $0\leq|P_{d}(t)-F_{(m)}(t)|\leq 1$ and the fourth inequality holds by triangle inequality. To prove (i), as $a\rightarrow\infty$, from James \cite{James}, $\displaystyle{\sup_{t\in\mathbb{R}}}| P_{d}(t)-H_{d}(t)|\xrightarrow{a.s.}0$ and by the continuous mapping theorem $\displaystyle{\sup_{t\in\mathbb{R}}}| H_{d}(t)-F_{(m)}(t)|\xrightarrow{a.s.}0$. To prove (ii), since $\displaystyle{\sup_{t\in\mathbb{R}}}| P_{d}(t)-H_{d}(t)|\xrightarrow{a.s.}0$ as $n\rightarrow\infty$ and $\mathcal{H}_{0}$ is true, the continuous mapping theorem and Polya's theorem \cite{Dasgupta} imply $\displaystyle{\sup_{t\in\mathbb{R}}}| H_{d}(t)-F_{(m)}(t)|\xrightarrow{a.s.}0$. Note that, the final results of part (i) and (ii) are concluded by practical assumptions in probability and measure theory given in Section 3.1 of Capi\'{n}ski and Kopp \cite{Capinski}.
\\
To prove (iii), note that, $\left(F_{(m)}(t)(1-F_{(m)}(t))\right)^{-1}\geq 4$, then
\begin{align*}
d_{AD}(P_{d},F_{(m)})&\geq 4\int_{-\infty}^{\infty}\left( P_{d}(t)-F_{(m)}(t)\right)^{2}\, dF_{(m)}(t)\\
&\geq 4\,d_{CvM}(P_{d},F_{(m)}).
\end{align*}
From Choi and Bulgren \cite{Choi}, since
\begin{small}
$
d_{CvM}(P_{d},F_{(m)})\geq \dfrac{1}{3}\left(\displaystyle{\sup_{t\in\mathbb{R}}}| P_{d}(t)-F_{(m)}(t)|\right)^{3}
$\end{small},
\begin{align*}
d_{AD}(P_{d},F_{(m)})&\geq \frac{4}{3}\left(\displaystyle{\sup_{t\in\mathbb{R}}}| P_{d}(t)-F_{(m)}(t)|\right)^{3}.
\end{align*}
Using the triangle inequality gives
\begin{align*}
\displaystyle{\sup_{t\in\mathbb{R}}}| P_{d}(t)-F_{(m)}(t)|\geq\displaystyle{\sup_{t\in\mathbb{R}}}| H_{d}(t)-F_{(m)}(t)|-\displaystyle{\sup_{t\in\mathbb{R}}}| P_{d}(t)-H_{(d)}(t)|.
\end{align*}
Similar to the proof of part (ii), as $n\rightarrow\infty$, $\displaystyle{\sup_{t\in\mathbb{R}}}| P_{d}(t)-H_{(d)}(t)|\xrightarrow{a.s.}0$ and
$\displaystyle{\sup_{t\in\mathbb{R}}}| H_{d}(t)-F_{(m)}(t)|\xrightarrow{a.s.}\displaystyle{\sup_{t\in\mathbb{R}}}| P_{true}(t)-F_{(m)}(t)|$, where $P_{true}$ is the true distribution of the sample $d$. Since $\mathcal{H}_{0}$ is not true, $\liminf\displaystyle{\sup_{t\in\mathbb{R}}}| P_{d}(t)-F_{(m)}(t)|\displaystyle{\overset{a.s.}{>}}0$, which implies $\liminf d_{AD}(P_{d},$ $F_{(m)})\displaystyle{\overset{a.s.}{>}}0$.
\\
To prove (iv), since for any \begin{small}$t\in \mathbb{R}$, $E_{P_{d}}\left(P_{d}(t)\right)=H_{d}(t)$\end{small} and \begin{small}$E_{P_{d}}\left( P_{d}(t)-H_{d}(t)\right)^{2}=\frac{H_{d}(t)\left( 1-H_{d}(t)\right)}{a+n+1}$\end{small}, then, as $n\rightarrow\infty$
\begin{small}
\begin{align*}\label{Exp of pos d_AD}
\liminf E_{P_{d}}\left( d_{AD}(P_{d},F_{(m)})\right)&\geq\liminf\int_{-\infty}^{\infty}\frac{H_{d}(t)\left( 1-H_{d}(t)\right)}{(a+n+1)F_{(m)}(t)\left( 1-F_{(m)}(t)\right)} dF_{(m)}(t)\nonumber\\
&+\liminf\int_{-\infty}^{\infty}\frac{\left(H_{d}(t)-F_{(m)}(t)\right)^{2}}{F_{(m)}(t)\left( 1-F_{(m)}(t)\right)}dF_{(m)}(t).
\end{align*}
\end{small}
Applying Fatou's lemma gives
\begin{small}
\begin{align*}
\liminf E_{P_{d}}\left( d_{AD}(P_{d},F_{(m)})\right)&\geq\int_{-\infty}^{\infty}\liminf\left(\frac{H_{d}(t)\left( 1-H_{d}(t)\right)}{(a+n+1)F_{(m)}(t)\left( 1-F_{(m)}(t)\right)}\right)dF_{(m)}(t)\nonumber\\
&+\int_{-\infty}^{\infty}\liminf\left(\frac{\left(H_{d}(t)-F_{(m)}(t)\right)^{2}}{F_{(m)}(t)\left( 1-F_{(m)}(t)\right)}\right)dF_{(m)}(t)\\
&=0+\int_{-\infty}^{\infty}\inf\left(\frac{\left(P_{true}(t)-F_{(m)}(t)\right)^{2}}{F_{(m)}(t)\left( 1-F_{(m)}(t)\right)}\right)dF_{(m)}(t).
\end{align*}
\end{small}
Since $\mathcal{H}_{0}$ is not true, $\inf\left(\frac{\left(P_{true}(t)-F_{(m)}(t)\right)^{2}}{F_{(m)}(t)\left( 1-F_{(m)}(t)\right)}\right)\displaystyle{\overset{a.s.}{>}}0$. Hence, by Theorem 15.2 of Billingsley \cite{Billingsley}, $\liminf E_{P_{d}}\left( d_{AD}(P_{d},F_{(m)})\right)\displaystyle{\overset{a.s.}{>}}0$.

\section{Chi-square Q-Q plot}

Chi-square Q-Q plot shows the relationship between the distribution of data points and chi-square distribution. It is a common graphical method to assess multivariate normality.  Specifically, for a given $m$-variate sampled data $\mathbf{y}_{m\times n}=\left(
\mathbf{y}_{1},\ldots,\mathbf{y}_{n}\right)$
with the size of $n$, where $\mathbf{y}_{i}\in\mathbb{R}^{m}$ for $i=1.\ldots,n$, chi-square Q-Q plot provide the plot of square ordered Mahalanobis distances against the quantile points of the chi-square distribution with $m$ degrees of freedom. From Johnson and Wichern \cite{Johnson}, if $\mathbf{y}_{m\times n}$ comes from an $m$-variate normal distribution, we should see the points of square ordered Mahalanobis distances are nearly a straight line with unit slope (reference line).

In the main paper,  we showed that the chi-square Q-Q plot confirms the result of the proposed test. For this, we provided the chi-square Q-Q plot with its approximate confidence envelope around the reference line for the data set of examples of Section 7 in the right side of Figure \ref{density-plot} and Figure \ref{QQ plot}. Note that, the confidence envelope is plotted to show the central intervals for each quantile of the chi-square distribution with $m$ degrees of freedom. In fact, the envelope in Figure \ref{density-plot} and \ref{QQ plot} provides a density estimate of the quantiles drawn from the chi-square distribution with  $m$ degrees of freedom. For more detailed studies about the general methods for Q-Q plot see Oldford \cite{Oldford}.

\end{appendices}


%
%



\end{document}